\documentclass{amsart}
\usepackage{amsmath,amssymb,epsfig,amscd,amsthm}

\usepackage{verbatim}

\newtheorem{thm}{Theorem}[section]
\newtheorem{lemma}[thm]{Lemma}
\newtheorem{conjecture}[thm]{Conjecture}
\newtheorem{cor}[thm]{Corollary}
\newtheorem{prop}[thm]{Proposition}

\newtheorem{claim}[thm]{Claim}

\numberwithin{equation}{thm}

\theoremstyle{remark}

\newtheorem{remark}[thm]{Remark}
\newtheorem{remarks}[thm]{Remarks}
\newtheorem*{notation}{Notation}
\newtheorem*{acknowledgments}{Acknowledgments}

\newcommand{\vlim}{\operatornamewithlimits{\text{$v$}-lim}}

\renewcommand{\bar}[1]{#1\llap{$\overline{\phantom{\rm#1}}$}}

\newcommand{\lra}{\longrightarrow}

\DeclareMathOperator{\hhat}{{\widehat{h}}}

\DeclareMathOperator{\gal}{{Gal}}
\DeclareMathOperator{\Res}{{Res}}
\DeclareMathOperator{\ord}{{ord}}
\DeclareMathOperator{\pic}{{Pic}}
\newcommand{\bA}{{\mathbb A}}
\newcommand{\N}{{\mathbb N}}

\newcommand{\Q}{{\mathbb Q}}

\newcommand{\C}{{\mathbb C}}

\newcommand{\PP}{{\mathbb P}}

\newcommand{\Kbar}{{\bar{K}}}
\newcommand{\Qbar}{\bar{\mathbb{Q}}}

\newcommand{\bC}{{\mathbb C}}

\newcommand{\Gal}{{\rm Gal}}
\newcommand{\OO}{{\mathcal O}}

\newcommand{\cF}{\mathcal{F}}

\newcommand{\cV}{\mathcal{V}}

\newcommand{\GCD}{\operatorname{GCD}}

\renewcommand{\l}{\lambda}

\renewcommand{\a}{\alpha}
\renewcommand{\b}{\beta}

\renewcommand{\d}{\delta}
\newcommand{\e}{\epsilon}

\newcommand{\bfa}{{\mathbf a}}
\newcommand{\bfb}{{\mathbf b}}
\newcommand{\bfc}{{\mathbf c}}

\newcommand{\bff}{{\mathbf f}}
\newcommand{\bfg}{{\mathbf g}}

\newcommand{\bfu}{{\mathbf u}}
\newcommand{\bfv}{{\mathbf v}}
\newcommand{\bfw}{{\mathbf w}}
\newcommand{\bfx}{{\mathbf x}}

\newcommand{\bfA}{{\mathbf A}}

\newcommand{\cL}{\mathcal{L}}
\newcommand{\cX}{\mathcal{X}}
\newcommand{\fo}{\mathfrak o}
\DeclareMathOperator{\Div}{div}
\newcommand{\bP}{{\mathbb P}}
\newcommand{\cO}{\mathcal{O}}

\begin{document}



\title{Preperiodic points for families of rational maps}

\author{D.~Ghioca}
\address{
Dragos Ghioca\\
Department of Mathematics\\
University of British Columbia\\
Vancouver, BC V6T 1Z2\\
Canada
}
\email{dghioca@math.ubc.ca}

\author{L.-C.~Hsia}
\address{
Liang-Chung Hsia\\
Department of Mathematics\\ 
National Taiwan Normal University\\
Taipei, Taiwan, ROC
}
\email{hsia@math.ntnu.edu.tw}

\author{T.~J.~Tucker}
\address{
Thomas Tucker\\
Department of Mathematics\\
University of Rochester\\
Rochester, NY 14627\\
USA
}
\email{ttucker@math.rochester.edu}

\thanks{The first author was partially supported by NSERC. The second
  author was partially supported by NSC Grant  99-2115-M-003-012-MY3 and
  he also acknowledges the support from NCTS.   The third author was
  partially supported by NSF Grants
  DMS-0854839 and DMS-1200749.}


\begin{abstract}
Let $X$ be a smooth curve defined over $\Qbar$, let $a,b\in\bP^1(\Qbar)$ and let $f_{\lambda}(x)\in \Qbar(x)$ be an algebraic family of rational maps indexed by all $\lambda\in
X(\C)$. We study whether there exist
infinitely many $\lambda\in X(\C)$ such that both $a$ and $b$ are
preperiodic for $f_{\lambda}$. In particular we show that if
$P,Q\in\Qbar[x]$ such that $\deg(P)\ge 2+\deg(Q)$, and if
$a,b\in\Qbar$ such that $a$ is periodic for $\frac{P(x)}{Q(x)}$, but
$b$ is not preperiodic for $\frac{P(x)}{Q(x)}$, then there exist at
most finitely many $\l\in\C$ such that both $a$ and $b$ are
preperiodic for $\frac{P(x)}{Q(x)}+\l$. We also prove a similar result
for certain two-dimensional families of endomorphisms of $\bP^2$. As a
by-product of our method we extend a recent result of Ingram
\cite{ingram10} for the variation of the canonical height in a family
of polynomials to a similar result for families of rational maps.
\end{abstract}


\maketitle


\section{Introduction}
\label{intro}

In \cite{Matt-Laura}, Baker and DeMarco study the following question:
given complex numbers $a$ and $b$, and an integer $d\ge 2$, when do
there exist infinitely many $\lambda\in\C$ such that both $a$ and $b$
are preperiodic for the action of $f_{\lambda}(x):=x^d+\lambda$ on
$\C$?  They show that this happens if and only if $a^d=b^d$.  The
problem, originally suggested by Zannier, is a dynamical analog of a
question on families of elliptic curves studied by Masser and Zannier
in \cite{M-Z-1, M-Z-2, M-Z-3}.  This problem was motivated by
the Pink-Zilber conjectures in arithmetic geometry regarding unlikely
intersections between a subvariety $V$ of a semiabelian variety $A$
and families of algebraic subgroups of $A$ of codimension greater than
the dimension of $V$ (see \cite{BMZ, Habegger, Pink}).  A thorough
treatment of the problem of unlikely intersections on families of
semiabelian varieties can be found in \cite{Zannier}.  
  
 The authors extended the results of \cite{Matt-Laura} to more general
families of polynomials in \cite{prep}.  
The polynomials  considered in \cite{Matt-Laura,  prep} are algebraic families 
parameterized by points in an affine subset of the projective line. In
this paper, we extend our investigation to families of rational
maps and the parameter spaces are general algebraic curves defined over a number field. Moreover, general families of
two-dimensional  endomorphisms of $\bP^2$ are also studied.
As in~\cite{Matt-Laura, prep}, a key ingredient in the study of
families of mappings is the application of equidistribution
theorems~\cite{Baker-Rumely, CL, FRL, Yuan} to the situation of arithmetic
dynamics. Note that the equidistribution results of \cite{Baker-Rumely, CL, FRL} apply only in dimension 1. As we also treat
the case of higher dimensional parameter spaces in this paper, we apply
the  equidistribution results, obtained by Yuan~\cite{Yuan} and the
  recent result of Yuan-Zhang~\cite{YZ} to these  more general
  families of maps.   We prove the following higher genus
  generalization of Theorem 1.1 of 
\cite{Matt-Laura}.  
 
\begin{thm}
\label{second main}
Let $C$ be a projective nonsingular curve defined over $\Qbar$, let $\eta\in C(\Qbar)$, and let $A$ be the ring of functions on $C$ regular on $C\setminus\{\eta\}$. Let $\Phi,\Psi\in A$ be nonconstant functions.   
Let $P_i,Q_i\in\Qbar[x]$ for $i=1,2$ be polynomials such that $\deg(P_i)\ge \deg(Q_i)+2$ for each $i$. Let 
$$\bff_{\l}(x):=\frac{P_1(x)}{Q_1(x)}+\Phi(\lambda)\text{ and }\bfg_{\l}(x):=\frac{P_2(x)}{Q_2(x)}+\Psi(\l)$$ 
be one-parameter families of rational maps  indexed by all $\lambda\in C(\C)$, and let $a,b\in\Qbar$ such that both $Q_1(a)$ and $Q_2(b)$ are nonzero. If there exist infinitely many $\lambda\in C(\C)$ such that $a$ is preperiodic under the action of $\bff_{\lambda}$ and $b$ is preperiodic under the action of $\bfg_\l$, then for \emph{each} $\l\in C(\C)$, $a$ is preperiodic under the action of $\bff_{\l}$ if and only if $b$ is preperiodic under the action of $\bfg_{\l}$.
\end{thm}

The following is an immediate consequence of Theorem~\ref{second main}.
\begin{cor}
\label{second main corollary}
Let $a,b\in\Qbar$, let $P_i,Q_i\in\Qbar[x]$ for $i=1,2$ be polynomials such that $\deg(P_i)\ge \deg(Q_i)+2$ for each $i$. Assume that $a$ is preperiodic for $\frac{P_1(x)}{Q_1(x)}$ but $Q_1(a)\ne 0$, and that $b$ is not preperiodic for $\frac{P_2(x)}{Q_2(x)}$. Then there exist at most finitely many $\l\in\C$ such that both $a$ is preperiodic for $\frac{P_1(x)}{Q_1(x)}+\l$ and also $b$ is preperiodic for $\frac{P_2(x)}{Q_2(x)}+\l$.
\end{cor}

Theorem~\ref{second main} will follow from our main result (see
Theorem~\ref{general main}). Our main result (see Theorem~\ref{general main}) also has applications
to the study of \emph{post-critically finite} maps. A map $f:\bP^1\lra \bP^1$ is
called post-critically finite (PCF) if each critical point of $f$ is
preperiodic under the action of $f$.  The post-critically finite maps
are important in algebraic dynamics; recently Baker and DeMarco \cite{BD-preprint} have
made a far-reaching conjecture about them. Baker and DeMarco are interested in locating the post-critically finite polynomials within the moduli space $\mathcal{P}_d$ of all polynomial maps of degree $d$.  Dujardin and Favre \cite{DF}
showed that the PCF maps are equidistributed with respect to the bifurcation measure
in $\mathcal{P}_d$; in particular, they form a Zariski dense subset of $\mathcal{P}_d$. Baker and DeMarco aim at
characterizing curves (or subvarieties) in $\mathcal{P}_d$
 containing a Zariski dense subset of PCF
maps; the expectation is that such subvarieties are very special.
Roughly speaking, \cite{BD-preprint} conjectures that a subvariety $V\subset \mathcal{P}_d$  contains a Zariski
dense subset of PCF maps if and only if $V$ is cut out by critical orbit relations. The
notion of ``critical orbit relation'' is a bit delicate, as one needs to take into account
the presence of symmetries in any given family of polynomials. We can
prove the following result which offers support to the main conjecture of \cite{BD-preprint}.
 
\begin{thm}
\label{PCF maps}
Let $f,g\in\Qbar[z]$ be polynomials of degree larger than $1$, let $C\subset \bA^2$ be a curve with the property that its projective closure in $\bP^2$ is a nonsingular curve with exactly one point at infinity. If there exist infinitely many points $(x,y)\in C(\Qbar)$ such that $f(z)+x$ and $g(z)+y$ are both PCF maps, then for each point $(x,y)\in C(\C)$, we have that $f(z)+x$ is PCF if and only if $g(z)+y$ is PCF.
\end{thm}

It is not clear, in general, how many of the results above should
carry over into higher-dimensional situations.  As another application
of the techniques developed in this paper, we are able to prove a
first result regarding \emph{unlikely intersections} for algebraic
dynamics in higher dimensions.

\begin{thm}
\label{first higher}
Let $P(X,Z)\in \Qbar[X,Z]$ and $Q(Y,Z)\in\Qbar[Y,Z]$ be homogeneous polynomials of degree $d\ge 3$ and assume that $P(X,0)$ and $Q(Y,0)$ are nonzero. Let $\bff_{\l,\mu}:\bP^2\lra \bP^2$ be the $2$-parameter family defined by
$$\bff_{\l,\mu}([X:Y:Z])=[P(X,Z)+\l YZ^{d-1}:Q(Y,Z)+\mu XZ^{d-1} :Z^d].$$
Let $a_i,b_i\in\Qbar^*$ (for $i=1,2$). If there exists a set of points $[\l:\mu:1]$ which is Zariski dense in $\bP^2$ such that for each such pairs $(\l,\mu)$  both $[a_1:b_1:1]$ and $[a_2:b_2:1]$ are preperiodic for $\bff_{\l,\mu}$ then for each $\l,\mu\in\Qbar$, $[a_1:b_1:1]$ is preperiodic for $\bff_{\l,\mu}$ if and only if $[a_2:b_2:1]$ is preperiodic for $\bff_{\l,\mu}$.
\end{thm}

We sketch briefly the ideas for proving Theorem~\ref{general
  main}. Let $\{\bff_{\l}\}$ be an algebraic family of rational maps
on the projective line $\bP^1$ parameterized by points $\l\in
Y(\Kbar)$ where $Y$ is an affine subset of the algebraic curve $X$
over a number field $K.$ As mentioned above, we apply recent results
of Yuan \cite{Yuan} and Yuan-Zhang \cite{YZ} to our situation.  The
main result of \cite{Yuan} shows that points of small height with
respect to a semipositive adelic metrized line bundle equidistribute
with respect to the measures induced by this semipositive adelic metrized
line bundle; the main result of \cite{YZ} says that when semipositive
metrics on a line bundle induce the same measures at a place, they
must differ by a constant. Taken together, these results say, roughly,
that if two appropriate height functions on a variety $X$ that come
from semipositive adelic metrized line bundles share a Zariski dense
family of points of small heights, then the two height functions must
be exactly the same. For a given family of points $\{\bfc_\l\}$, we
consider the canonical heights $\hhat_{\bff_\l}(\bfc_\l)$ of $\bfc_\l$
associated to the map $\bff_\l$.  A key observation is that under
appropriate conditions, a suitable multiple (depending on $\bfc$) of
$\hhat_{\bff_\l}(\bfc_\l)$ induces a height function $h_{\bfc}$ coming
from a metrized line bundle on $X$ (see
Section~\ref{subsec:specialization} for details). It follows that
$\bfc_\l$ is a preperiodic point for $\bff_\l$ if and only if
$h_{\bfc}(\l) = 0$.  Now, let $\bfc_1$ and $\bfc_2$ be two given
families of points.  For simplicity, here we only consider one family
of rational maps ($\bff_1 = \bff_2$ in the statement of the theorem)
on $\bP^1$ and assume that there are infinitely many $\l\in Y(\Kbar)$
such that both $\bfc_{1,\l}$ and $\bfc_{2,\l}$ are preperiodic for
$\bff_\l.$ Let $h_{\bfc_i}, i=1, 2$ be the corresponding heights on
$X$ induced from $\hhat_{\bff_\l}(\bfc_{i,\l}), i=1, 2$
respectively. Then, the infinite set of parameters $\l$ such that both
$\bfc_{1,\l}$ and $\bfc_{2,\l}$ are preperiodic points for $\bff_\l$
yields a Zariski dense set of small points on $X$. Using the results
of Yuan \cite{Yuan} and Yuan-Zhang \cite{YZ} mentioned above, we
conclude that the two height functions $h_{\bfc_1}$ and $h_{\bfc_2}$
are actually equal. Then we deduce that
$\hhat_{\bff_{\l}}(\bfc_{1,\l}) = 0$ if and only if
$\hhat_{\bff_{\l}}(\bfc_{2,\l}) = 0$ which concludes the proof of
Theorem~\ref{general main}.  Thus, our strategy follows that of
\cite{Matt-Laura}, but in the language of adelic metrized line bundles
rather than Green functions.

As a consequence of our method, using the notation from the previous paragraph, we prove that 
\begin{equation}
\label{important formula specialization}
\hhat_{\bff_\l}(\bfc(\l))=\hhat_\bff(\bfc)\cdot h_\bfc(\l)+O(1),
\end{equation}
where $\hhat_\bff(\bfc)$ is the canonical height of $\bfc\in \bP^1(F)$
under the action of $\bff:\bP^1\lra \bP^1$ (where $F=K(X)$). Formula
\eqref{important formula specialization} (which is proved in
Theorem~\ref{thm:uniform convergence}) extends a recent result of
Ingram \cite{ingram10} to certain families of rational maps. See Section~\ref{subsec:specialization} for the statement of Theorem~\ref{thm:uniform convergence} and a discussion of Ingram's result \cite{ingram10} (which in turn extends previous results of Silverman \cite{Silverman83} and Call-Silverman \cite{Call-Silverman}).

The plan of our paper is as follows. In Section~\ref{statements} we
state Theorem~\ref{general main} (which generalizes
Theorem~\ref{second main}) and also state few of its consequences.  We
also make a conjecture that would generalize the results of this paper
and place it in a more natural context.   Section~\ref{notation}
describes the notation that is used throughout the paper.  
Then, in Section~\ref{subsec:metrized line bundle} we introduce metrized line
bundles and state Yuan's \cite{Yuan} equidistribution result and
Yuan-Zhang's \cite{YZ} ``Calabi-Yau'' result on metrics that give rise
to the same measure. Section~\ref{equidistribution} is devoted to setting up our
problem so that the results from the previous section can be
applied. In Sections~\ref{specializations} and \ref{metrics} we prove
that our metrics satisfy the necessary hypotheses that allow us to use
the results of \cite{Yuan} and \cite{YZ}.   Section~\ref{conclusion}
contains a proof of
Theorem~\ref{general main} and of its consequences. In
Section~\ref{higher} we prove Theorem~\ref{first higher}. 

\begin{acknowledgments} 
Part of this paper was completed during the semester program ``Complex
and Arithmetic  Dynamics'' at ICERM in spring 2012. 
  The authors would like to thank ICERM for the support on
  participating the semester program. We also thank Matt Baker, Laura DeMarco, and Umberto Zannier for several conversations around this paper.
  
\end{acknowledgments}


\section{Statement of the main results}
\label{statements}


We make the following conjecture.  

\begin{conjecture}
\label{second conjecture}
Let $Y$ be any quasiprojective curve defined over $\Qbar$, and let $F$ be
the function field of $Y$.  Let $\bfa,\bfb\in\bP^1(F)$, and let
$V\subset \cX:=\bP^1_F\times_F\bP^1_F$ be the $\Qbar$-curve $(\bfa,\bfb)$.
Let $\bff:\bP^1\lra \bP^1$ be a rational map of degree $d \ge 2$
defined over $F$.  If there exists an infinite sequence of points $\l_n\in Y(\Qbar)$ such that 
$\lim_{n\to\infty}\hhat_{\bff_{\l_n}}(\bfa(\l_n)) = \lim_{n\to\infty}\hhat_{\bff_{\l_n}}(\bfb(\l_n))=0$, then  $V$ is contained in a proper
preperiodic subvariety of $\cX$ under the action of
$\Phi:=(\bff,\bff)$.
\end{conjecture}

Note that in the conjecture above, $\bff$ induces a well-defined
rational map $\bff_{\l}:\bP^1\lra \bP^1$ defined over $\Qbar$ for all
but finitely many $\l\in Y(\Qbar)$; as usual, $\hhat_{\bff_{\l}}$ is 
the (global) canonical height corresponding to the rational map $\bff_{\l}$.     
One may also phrase a ``Manin-Mumford''-type conjecture along the
lines of Conjecture ~\eqref{second conjecture}, which might hold for
preperiodic points over
$\bC$ (where one cannot define a height function) rather than
$\Qbar$. We note that one cannot extend the above Conjecture to actions of two arbitrary families $(\bff_\l,\bfg_\l):\bP^1\times\bP^1\lra \bP^1\times \bP^1$ (see the family of counterexamples from \cite{GT-IMRN}).

Recall that a point $P \in \PP^1$ is a superattracting
periodic point for $\bff$ if $P$ is  periodic of period $n$  for
$\bff$ and that $(\bff^n)'(P) = 0.$

\begin{thm}
  \label{general main}
  Let $K$ be a number field and let $X$
  be a smooth projective curve defined over $K$. Let $\eta\in X(K)$
  and let $Y:=X\setminus\{\eta\}$.  We let $\bfA$ be the ring of
  rational functions on $X$ which are regular on $Y$, defined over
  $K$. For each $\bfa\in \bfA$ we denote by $\deg(\bfa) =
  -\ord_{\eta}(\bfa)$ where $\ord_{\eta}$ is the order of the pole $\eta$ for the function $\bfa$.

Suppose that we have rational functions $\bff_1 = P_1(x)/Q_1(x)$ and
$\bff_2 = P_2(x)/Q_2(x)$ such that  $P_i, Q_i\in \bfA[x]$ and the
leading coefficients of $P_i$ and of $Q_i$ are  nonzero constants for
$i=1,2$. Furthermore, assume that $\bff_1$ and $\bff_2$ satisfy the following
conditions. 
\begin{enumerate} 
\item The resultant $R(\bff_i) := \Res(P_i(x), Q_i(x); x)$ of $P_i(x)$ and
  $Q_i(x)$ is a non-zero constant (i.e. $R(\bff_i) \in   K^{\ast}$). 
\item The point $x = \infty$ is a superattracting fixed point for both
  $\bff_1$ and $\bff_2.$
\end{enumerate}
For each $\l\in Y(\Kbar)$ and $i=1,2$, we denote by $\bff_{\l,i}$ the
rational function obtained by  evaluating  each coefficient of $P_i$
and of $Q_i$ at $\l$. We denote by $\hhat_{\bff_{\l,i}}$ the canonical
height associated to the function $\bff_{\l,i}$. 

Let $\bfc_i=\frac{\bfa_i}{\bfb_i}$ where
$\bfa_i, \bfb_i\in \bfA$ and
\begin{equation}
\label{relative prime again}
(\bfa_i,\bfb_i)=\bfA\text{ for }i=1,2,
\end{equation}
and and suppose that the two sequences $\{\deg(\bff_i^n(\bfc_i))\mid n
\in \N\}$ are not bounded.  If there exists an infinite family of
$\l_n\in Y(\Kbar)$ such that
$$\lim_{n\to\infty} \hhat_{\bff_{\l_n,1}}(\bfc_1(\l_n))=\lim_{n\to\infty} \hhat_{\bff_{\l_n,2}}(\bfc_2(\l_n)) =0,$$ 
then for all $\l \in Y(\Kbar)$, we have that $\hhat_{\bff_{\l,1}}(\bfc_1(\l))=0$  if and only if 
  $\hhat_{\bff_{\l,2}}(\bfc_2(\l))=0$.
\end{thm}

\begin{remarks}
  \begin{enumerate}
    \item[(1)] Condition~(2) in Theorem~\ref{general main} is
      equivalent to that $\deg_x P_i(x) \ge \deg_x Q_i(x) + 2$ for both
      $i =1,2.$ 
\item[(2)] Theorem~\ref{general main} yields that under the above hypotheses, $\bfc_1(\l)$ is preperiodic under $\bff_{\l,1}$ if and only if $\bfc_2(\l)$ is preperiodic under $\bff_{\l,2}$ since a point has canonical height equal to $0$ if and only it is preperiodic.
\item[(3)] We believe that Theorem~\ref{general main} should hold under more general hypotheses, i.e., the conclusion should still hold as long as neither $\bfc_1$ nor $\bfc_2$ is a (persistent) preperiodic point for $\bff_{1}$, respectively for $\bff_{2}$. In particular, Theorem~\ref{general main} should hold also for families of Latt\`{e}s maps associated to multiplication-by-$2$ on the elliptic curves $E_{\l}$, in which case one would establish a Bogomolov type result for the main theorems of Masser and Zannier from \cite{M-Z-1, M-Z-2, M-Z-3}.
\end{enumerate}
\end{remarks}

We have the following corollary (for more consequences, see our Section~\ref{conclusion}).

\begin{cor}
\label{cor 2}
Let $P_i,Q_i,R_i\in\Qbar[x]$ be nonconstant polynomials such that $\deg(P_i)>\deg(Q_i)+\deg(R_i)$ for $i=1,2$. Let $c_1,c_2\in\Qbar$ such that $c_1$ is preperiodic under the action of $P_1(x)/Q_1(x)$, while $c_2$ is not preperiodic under the action of $P_2(x)/Q_2(x)$. Then for any two nonconstant  polynomials $g_1,g_2\in\Qbar[x]$  such that $g_1(0)=g_2(0)=0$,   there exist at most finitely many $\l\in\Qbar$ such that both $g_1(\l)+c_1$ and $g_2(\l)+c_2$ are preperiodic under the actions of $P_1(x)/Q_1(x)+\l\cdot R_1(x)$, respectively of $P_2(x)/Q_2(x)+\l\cdot R_2(x)$.
\end{cor}

We present here a brief comparison of Theorem~\ref{general main} with the results from \cite{prep}. Firstly, our present method covers all the families of polynomials treated by the authors in \cite{prep}; however, it does not provide explicit relations between the starting points $\bfc_1$ and $\bfc_2$ as provided in \cite{prep} (see also \cite{Matt-Laura}).  The fact that in Theorems~\ref{second main} and \ref{general main} we do not obtain explicit relations between the starting points for the iterations (as obtained in \cite{Matt-Laura, prep})  is due to the fact that the analytic uniformization using Bottcher's Theorem (see \cite{Carleson-Gamelin}) cannot be used for giving an explicit formula for the local canonical height at an archimedean place for a rational map (which is \emph{not} totally ramified at infinity).


\section{Notation and Preliminaries}
\label{notation}

For any quasiprojective variety $X$ endowed with an endomorphism
$\Phi$, we call a point $x\in X$ \emph{preperiodic} if there exist two
distinct nonnegative integers $m$ and $n$ such that
$\Phi^m(x)=\Phi^n(x)$, where by $\Phi^i$ we always denote the
$i$-th iterate of the endomoprhism $\Phi$. If $x=\Phi^n(x)$ for some positive integer $n$, then $x$ is a \emph{periodic} point of \emph{period} $n$.

Let $K$ be a number field; we let $\Omega_K$ be the set of all absolute values  of $K$ which extend the (usual) absolute values of $\Q$. For each $v\in\Omega_K$, we let $v_0$ be the (unique) absolute value of $\Q$ such that $v|_\Q=v_0$ and we let $N_v:=[K_v:\Q_{v_0}]$. The (naive) Weil height of any point $x\in K$ is defined as
$$h(x)=\sum_{v\in\Omega_K}\frac{N_v}{[K:\Q]}\cdot \log\max\{1, |x|_v\}.$$
We will use the notation $\log^+(z)$ for $\log\max\{1,z\}$ for any real number $z$.

There exists a product formula for all nonzero elements $x$ of $K$, i.e., $$\prod_{v\in\Omega_K}|x|_v^{N_v}=1.$$

We fix an algebraic closure $\Kbar$ of $K$, and let $v \in \Omega_K$.
Let $\C_v$ be the completion of a fixed algebraic closure of the
completion of $(K,|\cdot |_v)$. When $v$ is an archimedean valuation,
then $\C_v=\C$.  We use the same notation $|\cdot |_v$ to denote the
extension of the absolute value of $(K_v,|\cdot |_v)$ to $\C_v$ and we
also fix an embedding of $\Kbar$ into $\C_v$.


Let $P=[x_0,\ldots, x_k]\in \PP^k(\Kbar)$ be given and let $P^{[1]}, 
\ldots, P^{[\ell]}$ denote the $\gal(\Kbar/K)$-conjugates of $P$. We let 
 $h_v(P) := \log \left(\max\{|x_0|_v, \ldots,|x_k|_v \}
 \right)$. Recall that the  Weil height of $P$ is given as follows.  
$$ h(P) := \frac{1}{\ell} \sum_{i=1}^\ell \sum_{v\in \Omega_K}\, N_v\,
h_v(P^{[i]}). 
$$
In this paper, we are primarily interested in points on the projective line ($k=1$). We fix an affine coordinate $z$ on $\PP^1$ and use the
identification $\PP^1(F) = F\cup \{\infty\}$ for any field $F.$ That
is, a point $x\in F$ is identified with the point $P = [x,1]\in \PP^1(F).$
The Weil height of $P = [x,1]$ is simply denoted by $h(x).$ 

Let $f\in K(x)$ be any rational map  of degree
$d\ge 2$.  In \cite{Call-Silverman}, Call and Silverman defined the \emph{global canonical height} $\hhat_f(x)$ for each $x\in\Kbar$ as
$$\hhat_f(x)=\lim_{n\to\infty} \frac{h(f^n(x))}{d^n}.$$
In addition, Call and Silverman proved that the global canonical height decomposes as a sum of the local canonical heights, i.e.
\begin{equation}
\label{decomposition for the global canonical height}
\hhat_f(x)=\frac{1}{[K(x):K]}\sum_{\sigma:K\lra \Kbar}\sum_{v\in\Omega_K}N_v \hhat_{f,v}\left(x^\sigma\right),
\end{equation}
where for each $v\in \Omega_{K}$ and for each $z\in\C_v$ we have
$$\hhat_{f,v}(z)=\lim_{n\to\infty}\frac{\log^+|f^n(z)|_v}{d^n}.$$
Using Northcott's Theorem one deduces that $x$ is preperiodic for $f$ if and only if $\hhat_f(x)=0$. This last statement does not hold if $K$ is a function field over a smaller field $K_0$ since $\hhat_f(x)=0$ for all $x\in K_0$ if $f$ is defined over $K_0$. 

We define heights in function fields similarly (see \cite{bg06, lang}). So, if $F$ is a function field of a projective normal variety $\cV$ defined over a field $K$ we denote by $\Omega_F$ the set of all absolute values on $F$ associated to the irreducible divisors of $\cV$. Then there exist positive integers $N_v$ (for each $v\in\Omega_F$) such that $\prod_{v\in\Omega_F}|x|_v^{N_v}=1$ for each nonzero $x\in F$. Also, we define the Weil height of any $P:=[x_0:\cdots :x_n]\in \bP^n(F)$ as 
$$h(P)=\sum_{v\in\Omega_F}N_v\cdot \log\left(\max\{|x_0|_v,\dots, |x_n|_v\}\right).$$
Following \cite{Call-Silverman}, we let the canonical height of $P$ with respect to an endomorphism $\varphi$ of $\bP^n$ of degree $d\ge 2$ be
$$\hhat_\varphi(P)=\lim_{n\to\infty}\frac{h(\varphi^n(P))}{d^n}.$$


\section{Heights and metrized line bundles}
\label{subsec:metrized line bundle}

Let $L$ be a line bundle on a nonsingular projective variety $X$ over a number field $K$ and let
$| \cdot |_v$ be an absolute value on $K$.  We say that $\| \cdot
\|_v$ is a {\it metric} on $L$ if 
$$\| (\alpha s)(P) \|_v = |\alpha|_v  \| s(P) \|_v$$
for any $P \in X(K_v)$ and any section $s$.  We say that ${\overline
  L}$ is an {\it adelic
metrized line bundle} over $K$ if it is equipped with a metric $\| \cdot
\|_v$ at each place $v$ of $K$.

When $\| \cdot \|_v$ is smooth and $v$ is
archimedean, we can form the {\it curvature} $c_1 ({\overline L})_v$ of $\| \cdot \|_v$ as
$$c_1 ({\overline L})_v = \frac{{\partial }{\bar \partial}}{\pi i} \log \| \cdot \|_v $$
on $X(\bC)$.  At the nonarchimedean places, Chambert-Loir \cite{CL}
has constructed an analog of curvature on $X^{an}_{\bC_v}$, using
methods from Berkovich spaces, in the case where the metric on the
line bundle is {\it algebraic} in the sense of being determined by the
extension of $L$ to a line bundle $\cL$ on a model $\cX$ for $X$
over $\fo_K$; that is, where $\| s(P) \|_v$ is determined by the
intersection of $\Div s$ with the Zariski closure of $P$ in $\cX$ at the place $v$.  

An adelic metrized line bundle ${\overline L}$ is said to be algebraic
if there is a model $\cX$ that induces the metric $\| \cdot \|_v$ at
each non-archimedean place.  An adelic metrized line bundle ${\overline
  L}$ is said to be {\em semipositive} (see \cite{Yuan, zhangadelic}) if there is a family of algebraic
adelic metrized line bundles ${\overline L_n}$ (with metrics denoted as $\| \cdot
\|_{v,n}$) such that:

\begin{enumerate}
\item at each $v$,  we have that $\log \| \cdot \|_{v,n}$ converges uniformly
  (over all of $X(K_v)$)
  to $\log \| \cdot \|_v$;
\item for each $n$ and each archimedean $v$, the metric $\| \cdot
  \|_{v,n}$ is smooth and the curvature of $\|
  \cdot \|_{v,n}$ is nonnegative; and
\item for all $n$, all nonarchimedean $v$, and any curve complete $C$
  on the model $\cX_n$ determining the metric $\| \cdot \|_v$ on $L_n$,
  the line bundle $\cL_n$ (described above) pulls back to a divisor of
  positive degree on $C$.
\end{enumerate}

In this case, one can assign a curvature $c_1 ({\overline L})_v$ to ${\overline L}$ at each
place $v$ by taking the limits of the curvatures of the metrics on ${\overline
  L}_n$.  

For any semipositive line bundle on a nonsingular subvariety $X$ and
any subvariety $Z$ of $X$, one can define a height $h_{{\overline
    \cL}}(Z)$ (see \cite{zhangadelic}).  In the case of points $x \in
X(\Kbar)$, with $\Gal(\Kbar/K)$-conjugates $x^{[1]}, \dots, x^{[\ell]}$
for
example, it is defined as 
\begin{equation}\label{height-def}
 \frac{1}{\ell} \sum_{i=1}^\ell \sum_{v\in\Omega_K} -N_v\cdot \log \| s(x^{[i]}) \|_v
\end{equation} 
where $s$ is a meromorphic section of $L$ with support disjoint from
the conjugates of $x$.

The following result states a fundamental equidistribution principle for points of small height on an adelic metrized line bundle which is pivotal for our proof.

\begin{thm}\label{YuanTheorem} \cite[Theorem 3.1]{Yuan} Suppose $X$ is a projective
  variety of dimension $n$ over a number field, and
  ${\overline L}$ is an adelic metrized line bundle over $X$ such that
  $L$ is ample and the adelic metric is semipositive. Let $\{x_m \}$ be an
  infinite sequence of algebraic points in $X ({\overline K})$ which
  is generic and small. Then for any place $v$ of $K$, the Galois
  orbits of the sequence $\{x_m \}$ are equidistributed in the
  analytic space $X^{an}_{\bC_v}$ with respect to the probability
  measure $d\mu_v = c_1 ({\overline L})_v^{n} / \deg_L( X )$.
\end{thm}

The next result we need can be stated for an individual metric $\|
\cdot \|_v$, where $v\in\Omega_K$ and $L$ is an ample
line bundle on a variety $X$ over $K_v$.  Recall that a metric $\|
\cdot \|_v$ is said to be semipositive (see \cite{Yuan, zhangadelic})  for 
archimedean $v$ when it is a uniform limit of smooth metrics meeting
condition (2) above and that it is semipositive for nonarchimedean $v$ when it is a
uniform limit of algebraic metrics meeting condition (3) above.

\begin{thm}\label{Calabi}\cite[Theorem 1.1]{YZ} Let $L$ be an ample line bundle
  over $X$, where $X$ is a projective variety over $K_v$, and let $\|
  \cdot \|_{v,1}$ and $\| \cdot |_{v,,2}$ be two  semipositive
  metrics on $L$. Then $c_1 (L, \| \cdot \|_{v,1})^{\dim X} = c_1 (L,
  \| \cdot \|_{v,2} )^{\dim X}$ if and only if $\frac{\| \cdot
    \|_{v,1}}{\| \cdot \|_{v,2}}$ is a constant.
\end{thm}

Combining Theorems~\ref{YuanTheorem} and \ref{Calabi}, we have the following result.  

\begin{cor}\label{to-use}
  Let $L$ be an ample line bundle on $X$ and let ${\overline L}_1$ and
  ${\overline L}_2$ be two semipositive adelic metrized line bundles
  over a number field $K$,
  each consisting of metrics on the same line bundle $L$.  Let $\{x_m
  \}$ be an infinite sequence of algebraic points in $X ({\overline
    K})$ that are Zariski dense in $X$. Suppose that
  \begin{equation}\label{small}
 \lim_{m\to\infty} h_{{\overline L}_1} (x_m) = \lim_{m\to\infty} h_{{\overline L}_2}(x_m) 
  = h_{{\overline L}_1} (X) = h_{{\overline L}_2}(X) = 0.
\end{equation} 
 Then $h_{{\overline L}_1} (z) = h_{{\overline
      L}_2}(z)$ for all $z \in X(\Kbar)$.
\end{cor}
\begin{proof}
  Note that this is implicit in \cite[Section 3]{YZ}, but for completeness,
  we give a proof. By Theorem~\ref{YuanTheorem}, the sequence $\{x_m\}$
  equidistribute with respect to both $c_1 ({\overline L}_1)_v^{n-1} /
  \deg_L( X )$ and $c_1 ({\overline L}_2)_v^{n-1} / \deg_L( X )$.
  Therefore, those two measures are the same, so the metrics
  are proportional. Since  $h_{{\overline L}_1} (z)$ and
  $h_{{\overline L}_2} (z)$ are computed by evaluating $-\log \| s \|_{v,1}$
  and $- \log \| s \|_{v,2}$ at $z$, it follows that $h_{{\overline
      L}_1}$ and $h_{{\overline L}_2}$ differ by a constant.  Since
  there is a sequence on which both converge to the same value, this
  constant must be zero.    
\end{proof}

As an example of family of metrics on the line bundle $L = \cO_{\bP^k}(1)$
on the projective space $\bP^k$, let $F_n: \bP^k \lra \bP^k$ be a
family of morphisms written
with respect to some coordinates $[X_0,\ldots, X_k]$: 
$$ F_n = [F_n^{[0]}: \cdots: F_n^{[k]}]$$
where each $F_n^{[i]}$ is a homogeneous polynomial of degree $e_n$ in $X_0,
\dots, X_n$. Then one might hope to metrize $L$ as
follows.  Let $s = a_0 X_0 + \dots + a_k X_k$ be a global section of
$L$.  Then at an archimedean place $v$, we define

\begin{equation}\label{1-arch}
 \| s(t_0, \dots, t_k) \|_{v,n} = \frac{ | a_0 t_0 + \dots + a_k t_k|_v}{  \left(
  |F_n^{[0]}(t_0, \dots, t_k)|_v^2 + \cdots + |F_n^{[k]}(t_0, \dots,
  t_k)|_v^2\right)^{1/(2e_n)}} 
\end{equation}
and at a nonarchimedean place $v$ we define

\begin{equation}\label{1-non-arch}
 \| s(t_0, \dots, t_k) \|_{v,n} = \frac{ | a_0 t_0 + \dots + a_k
   t_k|_v}{ \max\left\{|F_n^{[0]}(t_0, \dots, t_k)|_v, \cdots,
       |F_n^{[k]}(t_0, \dots, t_k)|_v\right\}^{1/e_n} }.
\end{equation}

 At archimedean places, for each $n$, we are essentially working with
 the Fubini-Study metric after pull-back by $F_n$ while at the
 nonarchimedean places we are working with the intersection metric
 after pull-back by $F_n$.  As long as the family of metrics
 $\|\cdot\|_{v,n}$ converges uniformly, their limit gives a
 semipositive metric on $L$.




\section{Family of rational maps and specializations}
\label{equidistribution}

In this section, we study a one parameter family of rational
maps. Several different height functions appear into the picture. We prove a
specialization theorem for these heights in the family of rational
maps in question. A similar result has been proved by Call and Silverman~\cite[Theorem~4.1]{Call-Silverman}.     
Using the method described in Section~\ref{subsec:metrized line bundle}, we are able to
give more precise information contained in the specialization
theorem.  Now let $K$ be a number field. We fix the following notation
throughout this section. 

\begin{notation}
\ \ 

\begin{itemize}
\item[$X$]
  a smooth, absolutely irreducible projective curve over $K$,
\item[$\eta$]
  a fixed $K$-rational point of $X$,
\item[$Y$]
  = $X \setminus \{\eta\}$,
\item[$F$]
  $ = K(X)$ the field of rational function on $X$,
\item[$\bfA$]
  $= \Gamma(\OO_X, Y) \subset F$ the ring of rational functions of $X$
  regular away   from $\eta$, 
\item[$u$]
  a uniformizer of $\eta$ (defined over $K$),
\item[$\deg(\cdot)$]
  $ = - \ord_{\eta}(\cdot).$ 
\end{itemize}
\end{notation}
By the definition of the degree function we have that 
$\deg(\bfa) \ge 0$ for all  $\bfa\in \bfA$.
Let  $\bfa \in \bfA$ be such that $\deg(\bfa) = n.$ Then the function
$g_{\bfa} := \bfa u^n$ has no pole at $\eta.$ We call the constant
$g_{\bfa}(\eta)$ the {\em leading coefficient} of $\bfa.$

\subsection{Family of rational  maps}
\label{other examples}

We consider a morphism 
$$\bff : \PP^1\to \PP^1 $$ of degree $d\ge 2$ over $F$ and
write $$\bff(x) = \frac{P(x)}{Q(x)} \quad \text{where} \;P(x),
Q(x) \in \bfA[x]
$$ 
such that  
$\GCD(P,Q) = 1$ ($P$ and $Q$ are viewed as elements in $F[x]).$
For ease of the notation, we put $d_P := \deg_x P(x)$ and $d_Q =
\deg_x Q(x).$ 
For a point $\l \in Y$, we use the following convention:
$$ P_{\lambda}(x) = \sum_{i=0}^{d_P} \bfc_{P, d_P-i}(\lambda)\,
x^i, \quad Q_\l(x)=\sum_{j=0}^{d_Q} \bfc_{Q, d_Q-j}(\l) \, x^j$$
and 
$$ \bff_{\lambda}(x) = \frac{P_{\lambda}(x)}{Q_\l(x)} \quad
\text{whenever $\bff_\l$ is well defined.}$$
Here, $\bfc_{P,i}, \bfc_{Q,j}\in \bfA$ are coefficients of $P(x)$ and
$Q(x)$ respectively.  
Thus, $\bff$ gives rise to a family of rational maps $\{\bff_\l\}$
parameterized by points $\l$ ranging over an affine open subset of $X.$
In the following, we work under the hypothesis of Theorem~\ref{general
  main}. Equivalently, $\bff$ satisfies the following conditions:  
\begin{enumerate}
\item[(1)] $d_P\ge d_Q+2$;
\item[(2)] the leading coefficients of both $P$ and $Q$ as polynomials
  in $x$ are constant; and
\item[(3)] the resultant $R(\bff)\in \bfA$ of $P(x)$ and $Q(x)$ is also a
  constant in $K^{\ast}$. 
\end{enumerate}


Condition~(1) yields that $d = \max\{d_P, d_Q\} =  d_P$ and $s =
d_P - d_Q \ge 2.$
Note that when $\bff(x)$ is a
polynomial we have $s =  d.$ In this case, we set
$Q(x) = 1$ and $\bfc_{Q,0} = 1.$ 
By abuse of the notation, we also write
$P(X,Y)$ and $Q(X,Y)$ for the homogeneous polynomials (in $X, Y$)
$Y^{d_P} P(X/Y)$, $Y^{d_Q} Q(X/Y)$ and let 
$$ \cF : \PP^1\to \PP^1\quad \cF([X,Y]) = [P(X,Y), Q(X,Y)]
$$
be a (fixed) representation of the given map $\bf$ in homogeneous
coordinates of $\PP^1.$ Using (3) above, it follows that
there exist polynomials $S, T, U, V \in \bfA[X,Y]$ homogenous in
variables $X, Y$ and positive integer $t \ge d$ such that
\begin{align}
  \label{resultant1}
S P + T  Q  & = X^t \quad \text{and}\\
\label{resultant2}
U P + V  Q  & = Y^t
\end{align}
where the homogeneous degrees (in $X$ and $Y$) are
$$\deg S = \deg T = \deg U = \deg V = t- d.$$

Let $\bfc:=\frac{\bfa}{\bfb}$ be a rational function on $X$,
where $\bfa, \bfb \in \bfA$. 
We assume that the ideal 
\begin{equation}
\label{relative prime ideal}
(\bfa, \bfb) = \bfA
\end{equation}
i.e., $\bfa$ and $\bfb$ are
relatively prime. 

Put 
\begin{equation}
\label{definition of the m's}
m_1  =\max_{i=1}^{d_P}\left\{\frac{\deg \bfc_{P,i}}{i}\right\}\text{
  and } \quad m_2=\max_{j=1}^{d_Q} \left\{\frac{\deg \bfc_{Q,j}}{j}\right\},
\end{equation}
and $m = m_1+m_2$.  By convention we set $m_2 = 0$ in the case where
$\bff$ is a polynomial map. The degree function $\deg$ on $\bfA$ has a
natural extension to $F$ given by $\deg(\bfc) = \deg(\bfa)-\deg(\bfb)$
for $\bfc\in F$ given above. In this section and the next, we
assume that $\bfc$ satisfies 
\begin{equation}
\label{assumption on d}
\deg(\bfc) > m .
\end{equation}
In particular, we know that $\deg(\bfc)\ge 1$ (since $m\ge 0$).  In
the following, we also set $d_{\bfa} := \deg(\bfa), d_{\bfb} :=
\deg(\bfb)$ and $d_{\bfc} := \deg(\bfc).$  

For each integer $n \ge 0$ we let $A_{\bfc,n}$ and $B_{\bfc,n}$ be
elements in $\bfA$ defined recursively as follows:
$$A_{\bfc,0}=\bfa\text{ and }B_{\bfc, 0}=\bfb,$$
while for all $n\ge 0$ we have
\begin{align}
\label{recursiveA}
A_{\bfc,n+1} &= P\left(A_{\bfc,n},B_{\bfc,n}\right)
\quad \text{and} \\
\label{recursiveB}
B_{\bfc,n+1}&
=Q(A_{\bfc,n},B_{\bfc,n}).
\end{align}

\begin{prop}
\label{degrees A and B 2}
For all $n\ge 1,$ we have that $(A_{\bfc,n}, B_{\bfc,n}) = \bfA$ and 
$\deg(A_{\bfc,n})=d_{\bfa}\cdot d^n$ while $\deg(B_{\bfc,n})=d_{\bfa}d^n -
d_{\bfc} s^n$.
Furthermore, the leading coefficient of $A_{\bfc,n}$ is
$c_{P,0}^{(d^{n}-1)/(d-1)}\cdot c_{\bfa}^{d^n}$, where 
$c_{\bfa}$ is the leading coefficient of $\bfa$.
\end{prop}
\begin{proof}
The computation of degrees is straightforward using \eqref{assumption
  on d}. The computation for the leading coefficient of $A_{\bfc,n}$
follows from the definition of the leading coefficient 
and induction on $n$. 

We show next that $A_{\bfc,n}$ and $B_{\bfc,n}$ are relatively prime polynomials. This assertion follows easily by induction on $n$; the case $n=0$ is immediate. 

Assume $A_{\bfc,n}$ and $B_{\bfc,n}$ are relatively prime for some
$n\ge 0$ and we prove 
that $A_{\bfc,n+1}$ and $B_{\bfc,n+1}$ are also relatively
prime. Substitute $X = A_{\bfc,n}$ and $Y = B_{\bfc,n}$ into
(\ref{resultant1}) and (\ref{resultant2}) and by the recursive
relations (\ref{recursiveA}) and (\ref{recursiveB}), we have
\begin{align*}
S(A_{\bfc,n}, B_{\bfc,n}) A_{\bfc,n+1} + T(A_{\bfc,n}, B_{\bfc,n})
B_{\bfc,n+1}  = A_{\bfc, n}^t\quad \text{and} \\
U(A_{\bfc,n}, B_{\bfc,n}) A_{\bfc,n+1} + V(A_{\bfc,n}, B_{\bfc,n})
B_{\bfc,n+1}  = B_{\bfc, n}^t . 
\end{align*}
It follows that the ideals $\left(A_{\bfc,n}^t,
  B_{\bfc,n}^t\right)\subseteq \left(A_{\bfc,n+1}, B_{\bfc,
    n+1}\right).$ Since by the induction 
hypothesis we have that $(A_{\bfc,n}, B_{\bfc,n}) = \bfA$, we 
conclude that $(A_{\bfc,n+1}, B_{\bfc,n+1}) = \bfA$. 
This completes the induction  and the proof of 
Proposition~\ref{degrees A and B 2}. 
\end{proof}
\begin{remark}
\label{degrees A and B 23} 
Note that by the definition of the iterates $\bff^n$, we have
$\bff^n(\bfc)=\frac{A_{\bfc,n}}{B_{\bfc,n}}$, and thus Proposition~\ref{degrees A and B 2} yields that $\bfc$ is not preperiodic under the action of $\bff$.
\end{remark}

The following is an easy corollary of Proposition~\ref{degrees A and B 2}.
\begin{cor}
\label{all in Kbar}
With the above notation, if $\bfc(\l)$ is preperiodic under the action
of $\bff_{\l}$, then $\l\in Y(\Kbar)$.
\end{cor}

\begin{proof}
If $\bfc(\l)$ is preperiodic under $\bff_{\l}$, then for some positive integers $n>m$ we have $\bff_{\l}^n(\bfc(\l))=\bff_{\l}^m(\bfc(\l))$. So,
$$\frac{A_{\bfc,n}(\l)}{B_{\bfc,n}(\l)}=\frac{A_{\bfc,m}(\l)}{B_{\bfc,m}(\l)},$$
where $A_{\bfc,n}$ and $B_{\bfc,n}$, and also $A_{\bfc,m}$ and
$B_{\bfc,m}$ are relatively prime. Thus $\l$ is a zero of the nonzero
rational function $A_{\bfc,n} B_{\bfc,m} - A_{\bfc,m} B_{\bfc,n}$ over
$K$, and hence $\l\in Y(\Kbar)$. 
\end{proof}

\subsection{Specialization theorem}
\label{subsec:specialization}

For a given rational map $\bff$ of
degree $d\ge 2$ over $F$ and a point $\bfc \in \PP^1(F)$, there are 
three heights: $\hhat_{\bff_\l}(\bfc(\l)), \hhat_{\bff}(\bfc)$ and
$h(\l)$. Namely, given $\l\in X(\Kbar)$ such that $\bff_\l$ is a
well-defined rational map $\bff_\l$ over $K(\l)$, the height
$\hhat_{\bff_\l}(\bfc(\l))$ is the canonical height  of
$\bfc(\l)$ associated to $\bff_\l$ and 
$h(\l)$ is a Weil height associated to a degree one divisor class
on $X$; while $\hhat_{\bff}(\bfc)$ is the canonial height of $\bfc$  
associated to $\bff$ over $F.$ Call and 
Silverman~\cite[Theorem~4.1]{Call-Silverman} have shown that
\begin{equation}
  \label{eq:height specialization}
  \hhat_{\bff_\l}(\bfc(\l)) = \hhat_{\bff}(\bfc) h(\l) + o(h(\l))
  \end{equation}
which generalizes a result of Silverman~\cite{Silverman83} on heights 
of families of abelian varieties. In a recent paper~\cite{ingram10},
Ingram shows that for a family of polynomial maps $\bff\in F[x]$ and $P \in
\PP^1(F)$, there is a divisor $D = D(\bff, P) \in \pic(X)\otimes \Q$ of
degree $\hhat_{\bff}(P)$ such that
$$
 \hhat_{\bff_\l}(P_\l) = h_D(\l) + O(1).
$$
This result is an analogue of Tate's theorem~\cite{tate83} in  the
setting of arithmetic dynamics. Using this result and applying an
observation of Lang, the error term in~\eqref{eq:height
  specialization} is improved to $O(h(\l)^{1/2})$ and 
furthermore, in the special case where $X = \PP^1$ the error term
can be replaced by  $O(1)$~\cite[Corollary 2]{ingram10}.  

In order to apply Theorem~\ref{Calabi} to our situation, the
error term in~\eqref{eq:height specialization} needs to be controlled within
$O(1).$  In general, this may not be true without further restrictions
on $\bff$ and $P$. Ingram's result shows that this is true if $\bff$
is a polynomial map and the parameter space $X = \PP^1$. 
In this paper, we provide another set of conditions for $\bff$ and the
point $P\in \PP^1(F)$ so  that the error term in~\eqref{eq:height
  specialization} is  $O(1).$ 

\begin{thm}
  \label{thm:uniform convergence}
  Let $\bff(x) := P(x)/Q(x) \in F(x)$ be of degree $d \ge 2$ over $F$
  and assume that  $\bff$  satisfies the following conditions
\begin{enumerate}
\item the resultant $R(\bff)$  and the leading coefficients of $P(x)$
  and $Q(x)$ are nonzero constants;

\item  the point $ x = {\mathbf \infty}$ is a superattracting
  periodic   point for $\bff$.  
\end{enumerate}
Let $\bfc\in\PP^1(F)$ be such that the sequence
$\{\deg(\bff^n(\bfc))\}_{n\ge 0}$ is unbounded. Then for $\l \in
Y(\Kbar)$ we have
$$ \hhat_{\bff_\l}(\bfc(\l))  = \hhat_{\bff}(\bfc) h(\l) + O(1) $$
where $h$ is a height function associated to the divisor class
containing the divisor $\eta.$ 
\end{thm}

\begin{remark}
\label{important remark Ingram}
We actually show that the function $\hhat_{\bff_\l}(\bfc(\l))/
\hhat_{\bff}(\bfc)$ is a height function coming from a metrized
line bundle on $X$. In the case proved by Ingram that is not covered by our
theorem, i.e.  the case where $\bff$ is a polynomial with parameter space $X =
\bP^1$ and  $\bfc\in \bP^1(F)$ without any further restriction, it would
be interesting to see whether or not 
$\hhat_{\bff_\l}(\bfc(\l))/\hhat_{\bff}(\bfc)$ also gives rise to a
height function coming from a metrized line bundle on $\bP^1.$ 
\end{remark}

Theorem~\ref{thm:uniform convergence} will follow from
Proposition~\ref{finally uniformity} proved below. The proof of
Theorem~\ref{thm:uniform convergence} follows 
the idea described in Section~\ref{subsec:metrized line bundle} and
will be given later. The following two sections are devoted to the
proof of Proposition~\ref{finally uniformity}. 




\section{Growth of the iterates in fibers above $X$}
\label{specializations}

We continue with the notation from the previous Section.

Recall that we have fixed a uniformizer $u$ of $\eta$. Then 
there exists a Zariski open neighborhood $Z$ of
$\eta$  such that the uniformizer $u$ is a regular function on
$Z.$ We fix such a neighborhood of $\eta.$ 
For a given place $v\in\Omega_K$, $Y(\C_v)$ has a topology called the 
$v$-adic topology induced by the absolute value $|\cdot|_v$ on $\C_v.$
Each $\bfa\in \bfA$ yields a continuous function $\bfa : Y(\C_v) \to
\C_v$ with respect to the $v$-adic topology. For any large $L > 0$, we let $V_{L,v} \subset Z(\C_v)$ be the $v$-adic open neighborhood of $\eta$ containing all points $\l\in Z(\C_v)$ such that $|u(\l)|_v<\frac{1}{L}$. If
there is no danger of confusion, we drop the subscript $v$ below. 
Denote the complement of $V_L$ by $U_L := X(\C_v)\setminus V_L \subset
Y(\C_v).$ It follows that $\bfa$ is 
bounded on $U_L.$ Let $n = \deg(\bfa)$ and put $g_{\bfa} = \bfa
u^n$. By increasing $L$ if necessary, we may assume that $g_{\bfa}$
is bounded on $V_L$; let $C>0$ be an upper bound for $|g_\bfa|_v$ on $V_L$.  Thus for each $\l$ in the boundary of $V_L$ we have that $|\bfa(\l)|_v\le C L^n$. Furthermore, for $L$ sufficiently large, the maximum of $|\bfa(\l)|_v$ on $U_L$ is attained on the boundary of $U_L$ (which is the same as the boundary of $V_L$) and thus $|\bfa(\l)|_v\le C L^n$ for all $\l \in U_L$. Note that even though apriori, $C$ depends on $L$ (and also on $\bfa$ and $v$), the dependence on $L$ is not essential since once we replace $L$ with a larger number $L'$, the same value of $C$ would work as an upper bound for $g_\bfa$ on $V_{L'}$ because $V_{L'}\subset V_L$ (this fact will be used in our proof). More generally, for a
nonempty finite subset $T$ of $\bfA$ and large $L$, by shrinking $V_L$
if necessary, we 
may assume that there exists a positive constant $C$ depending only on $T$ and $v$ such that 
the inequality $|\bfa(\l)|_v \le C L^n$ holds for all $\bfa\in T$ and $\l\in
U_L$ where $n = \max_{\bfa\in T}\deg(\bfa)$. Furthermore, for any
polynomial $g(x)\in \C_v[x]$ we define  $|g|_v$ to be the maximum 
of the $v$-adic norms of its coefficients.  

\begin{prop}
 \label{lem:finite uniform}
Let $v\in\Omega_K$ be any place, and let
$M_n(\lambda):=\max\{|A_{\bfc,n}(\lambda)|_v,
|B_{\bfc,n}(\lambda)|_v\}$ for each $n\ge 0$ and $\l\in Y(\C_v)$. Let
$L\ge 1$ be a large positive number and let $U_L\subset Y(\C_v)$ be
determined as above. 
Then, there exist positive constants $C_1, C_2$ depending only on 
$v$, $L$ and on the coefficients $\bfc_{P,i}$ of $P$ and $\bfc_{Q,j}$
of $Q$  such that for all $n\ge 0$ we have
$$
   C_1 M_n(\lambda)^{d} \le  M_{n+1}(\lambda) \le C_2 M_n(\lambda)^{d}
$$
for all $\lambda\in U_L .$
\end{prop}

\begin{proof}
Let $\lambda\in U_L$ be given. 
The proof uses standard techniques in height theory. Before we deduce
the upper bound in the above inequalities, we first note that
$$ |P_{\lambda}|_v = \max\{|c_{P,i}(\lambda)|_v : i=0,\ldots, d_P\}\text{ and }|Q_{\lambda}|_v = \max\{|c_{Q,j}(\lambda)|_v : j=0,\ldots, d_Q\}.$$
Therefore there exists a constant $C_3$ depending only on $L$ and on the coefficients of $P_\l$  such
that  $|P_{\lambda}|_v \le C_3 L^{m_1d_P}$ and $|Q_\l|_v\le C_3 L^{m_2d_Q}$ for all $\lambda\in U_L$.
For any integer $k$, we  use the following
notation 
$$\e_v(k) = \begin{cases} k & \text{if $v$ is archimedean,} \\
                          1 & \text{if $v$ is nonarchimedean.}
\end{cases}$$
By (\ref{recursiveA}) and (\ref{recursiveB}), we have that 
\begin{align*}
|A_{\bfc,n+1}(\l)|_v & =  |P_\l\left(A_{\bfc,n}(\lambda),B_{\bfc,n}(\lambda)\right)|_v
\\
& \le   \left(\e_v(d_P+1) |P_\l|_v\right) M_n(\lambda)^{d_P}
\\
& \le \left(\e_v(d_P+1) C_3 L^{m_1d_P} \right) M_n(\lambda)^{d_P}
\end{align*}
and
\begin{align*}
|B_{\bfc,n+1}(\lambda)|_v & \le 
|B_{\bfc,n}(\lambda)^{s}Q_\l(A_{\bfc,n}(\lambda),B_{\bfc,n}(\lambda))|_v \\
& \le \left(\e_v(d_Q+1)C_3 L^{m_2d_Q}\right) M_n(\lambda)^{d_P}.
\end{align*}
So, the right-hand side inequality from the conclusion of
Proposition~\ref{lem:finite uniform} holds with $C_{2} = \e_v(d_P+1) C_3\left(L^{m_1d_P}+L^{m_2d_Q}\right)$, for example. 

Next, we deduce a complementary inequality. For this, we substitue
$X = A_{\bfc,n}$ and $Y = B_{\bfc,n}$ into  (\ref{resultant1}) and
(\ref{resultant2}). Then, as in the proof of Proposition~\ref{degrees
  A and B 2} we have 
\begin{align*}
S(A_{\bfc,n},B_{\bfc,n})  A_{\bfc,n+1} +
T(A_{\bfc,n},B_{\bfc,n}) B_{\bfc,n+1} & = A_{\bfc,n}^t
\quad \text{and} \\
U(A_{\bfc,n},B_{\bfc,n})  A_{\bfc,n+1} +
V(A_{\bfc,n},B_{\bfc,n})  B_{\bfc,n+1}& =  B_{\bfc,n}^t.  
\end{align*}
Note that, as polynomials in variables $X$ and $Y$, the coefficients of
$S, T, U$ and $V$ are in $\bfA.$ Let the maximal degrees of
coefficients  of $S, T, U$ and $V$ be $\ell$. Then for  $\lambda\in U_L$ there exists a positive real  constant $C_4$ such that
$$ \max\{|S|_v, |T|_v, |U|_v, |V|_v\} \le C_4 L^{\ell}.$$
Applying triangle inequality, we have
\begin{align*}
| A_{\bfc,n}(\lambda)|_v^t & \le \e_v(t-d_P+1)|S|_v
M_n(\lambda)^{t-d_P} M_{n+1}(\l) +  \e_v(t-d_P +1) |T|_v
M_n(\lambda)^{t-d_P}M_{n+1}(\l)   \\
& \le 2 \e_v(t-d_P+1)  C_4 L^{\ell} M_n(\l)^{t-d_P} M_{n+1}(\lambda) 
\end{align*}
and similarly,
\begin{align*}
|B_{c,n}(\lambda)|_v^t & \le  2 \e_v(t-d_P+1)  C_4 L^{\ell} M_n(\l)^{t-d_P} M_{n+1}(\lambda) 
\end{align*}
Hence,
$$ 
M_n(\lambda)^t \le 2 \e_v(t-d_P+1)  C_4 L^{\ell}
M_{n}(\lambda)^{t-d_P} M_{n+1}(\lambda) 
$$
and thus the desired lower bound from Proposition~\ref{lem:finite
  uniform} is obtained by taking   $C_1 = 1/(2 \e_v(t-d_P+1)  C_4
L^{\ell})$ and note that $d_P = d$. 
\end{proof}

Next, we fix  $v\in\Omega_K$ and show that $M_n(\lambda)$ also
satisfies similar relations as 
stated in Proposition~\ref{lem:finite uniform} for $\l \in V_L$ with
$L$ large enough. In the following the notation $\vlim_{\l\to
  \eta}$ means that the limit is taken for the point $\l$
approaching $\eta$ with respect to the $v$-adic topology. 
We first observe that
$$ \vlim_{\lambda \to \eta}
\frac{|P_{\lambda}(\bfc(\lambda))|_v}{|\bfc(\lambda)|_v^{d_P}} = |c_{P,0}|_v \quad
\text{and}\quad  \vlim_{\lambda \to \eta}
\frac{|Q_\l(\bfc(\lambda))|_v}{|\bfc(\lambda)|_v^{d_Q}} = |c_{Q,0}|_v .$$ 
Indeed, the assertions follow from  the choice of $d_\bfc =\deg\bfc$ such that 
$i d_\bfc > \deg c_{P,i}$ for $i = 1, \ldots, d_P $ and $jd_\bfc>\deg c_{Q,j}$ for $j=1,\ldots,d_Q$. 
Furthermore,  we have $\vlim_{\lambda\to \eta}
|\bfc(\lambda)|_v|u(\lambda)|_v^{d_\bfc} = |c_{\bfa}/c_{\bfb}|_v$ and
$\vlim_{\lambda\to \eta}
|f_{\lambda}(\bfc(\lambda))|_v/|\bfc(\lambda)|_v^s =
\left|c_{P,0}/c_{Q,0}\right|_v.$
It follows that there exist positive real numbers $L_1>1$   and $\d_1 < 1$ 
such that for all $\lambda \in V_{L_1}\setminus\{\eta\}$ we have
\begin{align}
  \label{asymptotic f}
\d_1 |\bfc(\lambda)|_v^s \le   & |f_{\lambda}(\bfc(\lambda))|_v \le \frac{1}{\delta_1} |\bfc(\lambda)|_v^s 
  \\
\label{asymptotic bfc}
  \delta_1 |z(\lambda)|_v^{d_\bfc} \le  & |\bfc(\lambda)|_v \le
  \frac{1}{\delta_1} |z(\lambda)|_v^{d_\bfc},
\end{align}
where, for convenience, we set $z = 1/u$ so that $z(\l) = 1/u(\l)$ for $\l\in
V_{L}\setminus\{\eta\}.$ Furthermore, without loss of generality, we may assume that $\d_1<\frac{|c_{P,0}|_v}{2|c_{Q,0}|_v}$. 

\begin{lemma}
  \label{estimate for f}
  Let   $L_2\ge L_1$ be a real number.  Then there exists a real number
  $L_3\ge L_2$   such that for all $\bfx\in F$ satisfying
  $\left|\bfx(\l)/\bfc(\l)\right|_v>2$ for all $\l\in V_{L_2}\setminus \{\eta\}$, we have
  $$ |f_{\lambda}(\bfx(\lambda))|_v >   \delta_1 |\bfx(\lambda)|_v^s $$
  for all $\lambda\in V_{L_3}\setminus\{\eta\}.$
\end{lemma}

\begin{proof}
 Let the  Laurent series expansion in $x^{-1}$ of   $f_\l(x)$
 be as follows
 $$ f_\l(x) = \frac{P_\l(x)}{Q_\l(x)} = 
\sum_{k \ge 0} \a_{k} x^{s-k} $$
where $\a_k \in \bfA$ and $\a_0 (\lambda)=
c_{P,0}/c_{Q,0} \in K^{\ast}$. We estimate next $|\a_k(\l)|_v$ for $\l \in V_{L_2}\setminus\{\eta\}$ as $k$ varies. For this  we write $x^{d_Q}Q_\l(x)^{-1} = \sum_{j\ge 0} \b_j x^{-j}\in
\bfA[[x^{-1}]]$ and we claim that there exist positive real numbers $C_5$ and $C_6$ such that
\begin{equation}
\label{growth for beta}
|\b_j(\l)|_v \le C_5\left(C_6 |z(\l)|_v^{m_2}\right)^j.
\end{equation}
Indeed, \eqref{growth for beta} follows from the fact that $\b_0=1/c_{Q,0}$ while for each $j\ge 1$ we have
$$-c_{Q,0}\b_j=\sum_{\substack{i_1+i_2=j\\1\le i_1\le d_Q\\ 0\le
    i_2\le j-1}} c_{Q,i_1} \,\b_{i_2}.$$
So, using that $|c_{Q,i}(\l)|_v=O\left(|z(\l)|_v^{m_2i}\right)$, 
an easy induction finishes the proof of \eqref{growth for beta}. On
the other hand, we have 
$$\a_k(\lambda) = \sum_{\substack{i_1+i_2=k\\0\le i_1\le d_P\\ 0\le i_2\le d_Q}}
\b_{i_2}(\l) c_{P,i_1}(\lambda) .$$
In particular, $\a_0=c_{P,0}/c_{Q,0}$. Using \eqref{growth for beta} coupled with the fact that $|c_{P,i}(\l)|_v=O\left(|z(\l)|_v^{m_1i}\right)$ we get that there exist positive real numbers $C_7$ and $C_8$ (independent of $k$) such that
$$|\a_k(\l)|_v\le C_7 \left(C_8|z(\l)|^{m_1+m_2}_v\right)^k\text{ for all }k.$$
Since $|\bfc(\l)|_v\ge \d_v |z(\l)|_v^{d_\bfc}$ and $d_\bfc> m = m_1+m_2$ we obtain that for sufficiently large $|z(\l)|_v$ we have
$\max\{|\a_k(\lambda)\bfc(\lambda)^{-k}|_v \text{ : } k \in \N\} < |\a_0|_v/2$  and furthermore if $v$ is
archimedean, then $\sum_{k\ge 1}|\a_k(\lambda)\bfc(\lambda)^{-i}|_v^2$
is convergent and bounded above by  $ |\a_0|_v^2/4$.

We let $L_3>L_2$ be a sufficiently large real number  such that for
$|z(\lambda)|_v > L_3$ we have 
$$
    \sum_{k\ge 1}     |\a_k(\l)\bfc(\lambda)^{-k}|_v^2 <  |\a_0|_v^2/4  \text{ if $v$
      is archimedean,} 
 $$
and
$$
 \max\{\left|\a_k(\lambda)\bfc(\lambda)^{-k}\right|_v \text{ : } k \ge 1\} <
 \frac{|\a_0|_v}{2}  \text{ if $v$ is non-archimedean.}
$$

Now let $\lambda\in V_{L_3}\setminus\{\eta\}$ and $\bfx\in F$ such that
$|\bfx(\l)/\bfc(\l)|_v > 2$. Write 
$$ \sum_{k \ge 1} \a_k(\l) \bfx(\lambda)^{-k} = \sum_{k\ge 1}
\a_k(\lambda) \bfc(\lambda)^{-k}
\left(\frac{\bfx(\lambda)}{\bfc(\lambda)}\right)^{-k}. $$ 

If $v$ is non-archimedean, then we have
$$\left| \sum_{k \ge 1} \a_k(\l) \bfx(\lambda)^{-k}\right|_v \le 
\max\left\{\left|\a_k(\lambda) 
\bfc(\lambda)^{-k}\left(\frac{\bfx(\lambda)}{\bfc(\lambda)}\right)^{-k}\right|_v\text{ : }
k\ge 1\right\} \le \frac{|\a_0|_v}{2}.$$ Therefore,
$$ \left|\frac{f_{\lambda}(\bfx(\l))}{\bfx(\l)^s}\right|_v = \left|\a_0 + \sum_{k\ge
    1}\a_k(\lambda) \bfx(\l)^{-k}\right|_v =  |\a_0|_v  \ge \d_1.$$

Now assume  $v$ is archimedean. By the choice
of $L_3$, the two sequences of complex numbers
$\left(\a_k(\lambda)\bfc(\lambda)^{-k}\right)_{k\ge 1}$ and  $\left(
  (\bfx(\lambda)/\bfc(\lambda))^{-k}\right)_{k\ge 1}$ are both square
summable for each $\l\in V_{L_3}\setminus\{\eta\}$. Hence, by the Cauchy-Schwartz inequality  we see that
\begin{align*}
  \left|\sum_{k\ge 1} \a_k(\lambda)\bfc(\lambda)^{-k}\left( \frac{\bfx(\lambda)}{\bfc(\lambda)}\right)^{-k} \right|_v^2 & \le \left(\sum_{k\ge
      1} |\a_k(\lambda)\bfc(\lambda)^{-k}|_v^2 \right)
  \left(\sum_{k\ge1}\left|\frac{\bfx(\lambda)}{\bfc(\lambda)}\right|_v^{-2k} \right)
  \\
  & \le \frac{|\a_0|_v^2}{4}\cdot \left(\sum_{k\ge1} \frac{1}{4^k}\right) \le \frac{|\a_0|_v^2}{4}. 
\end{align*}
Hence,
$$ \left|\frac{f_{\lambda}(\bfx(\l))}{\bfx(\l)^s}\right|_v \ge \left| |\a_0|_v -
  \sqrt{\frac{|\a_0|_v^2}{4}}\right| \ge \frac{|\a_0|_v}{2}.$$ 
In both cases, we have shown that for all $\lambda\in V_{L_3}$ 
we have $|f_{\lambda}(\bfx(\lambda))|_v \ge \d_1
|\bfx(\lambda)|_v^s$ as desired.  
\end{proof}

\begin{prop}
\label{lem:infinite uniform}
There exists a number $L_3\ge L_1$ depending only on the coefficients
of $P_\l$, $Q_\l$ (and on $L_1$)  such that for all $n \in \N$ and all $\lambda\in V_{L_3}\setminus\{\eta\}$
we have 
\begin{equation}
\label{5 22}
|f_\lambda^n(\bfc(\l))|_v \ge \d_1^{(s^n-1)/(s-1)} |\bfc(\l)|_v^{s^n}
\ge \d_1^{(s^{n+1}-1)/(s-1)}  |z(\l)|_v^{d_\bfc s^{n}}. 
\end{equation}
\end{prop}

\begin{proof}
Firstly, we note that if $|z(\l)|_v> L_3\ge L_1$, then \eqref{asymptotic bfc} yields the second inequality from \eqref{5 22}. So, we are left to prove the first inequality in \eqref{5 22}.

We claim that there exists a real number $L_2$ larger than $L_1$ which also satisfies the following properties:
\begin{enumerate}
\item[(a)] if $\l\in V_{L_2}\setminus\{\eta\}$, then $\d_1 |\bfc(\l)|_v^{s-1} >  2.$
\item[(b)] if $\l\in V_{L_2}\setminus\{\eta\}$, then $|f_\l(\bfc(\l))|_v \ge\d_1 |\bfc(\l)|_v^s$.
\end{enumerate} 
We can obtain inequality (a) above since if $|z(\l)|_v>L_2\ge L_1$, then 
\begin{equation}
\label{asymptotic bfc 2}
|\bfc(\l)|_v\ge \d_1 |z(\l)|_v^{d_\bfc}>\d_1 L_2^{d_\bfc},
\end{equation}
and thus if $L_2>\left(\frac{2}{\d_1^s}\right)^{1/d_\bfc (s-1)}$, inequality (a) is satisfied.  
In order to  obtain inequality (b), using \eqref{asymptotic bfc 2}, it suffices to choose $L_2$ satisfying the inequality $\d_1 L_2^{d_\bfc}>L_1$ (or equivalently, $L_2>(L_1/\d_1)^{1/d_\bfc}$). Then we may employ \eqref{asymptotic f} and obtain inequality (b) above.

We  let $L_3\ge L_2$ be the real number satisfying the conclusion of Lemma~\ref{estimate for f}.   
Let $\lambda \in V_{L_3}\setminus\{\eta\}.$ The proof of the first inequality in the
conclusion of Proposition~\ref{lem:infinite uniform} is by induction
on $n \ge 1.$ The inequality 
\eqref{5 22} for $n=1$ is precisely inequality (b) above.

Next we prove the inductive step; so we assume \eqref{5 22} holds for some $n\ge 1$ and we will prove that
$$
\left|f_\lambda^{n+1}(\bfc(\l))\right|_v \ge \d_1^{(s^{n+1}-1)/(s-1)}
|\bfc(\l)|_v^{s^{n+1}}.
$$

By induction hypothesis, we know that 
$$
\left|f_\lambda^n(\bfc(\l))\right|_v \ge \d_1^{(s^n-1)/(s-1)}
|\bfc(\l)|_v^{s^n} \ge \d_1^{(s^{n+1}-1)/(s-1)} |z(\l)|_v^{d_\bfc s^n}. 
$$
We shall apply Lemma~\ref{estimate for f} to $\bfx(\l) =
f_\l^n(\bfc(\l))$. In order to do this we need to check  that
\begin{equation}
\label{hypothesis check}
\left|f_\l^n(\bfc(\l))/\bfc(\l)\right|_v > 2\text{ if } \l\in V_{L_2}\setminus\{\eta\}.
\end{equation}
Indeed, we notice that
\begin{align*}
\left|\frac{f_\l^n(\bfc(\l))}{\bfc(\l)}\right|_v & \ge
\d_1^{(s^n-1)/(s-1)} |\bfc(\l)|_v^{s^n - 1}\text{ by the inductive hypothesis}
\\
& \ge \left( \d_1  |\bfc(\l)|_v^{s-1}\right)^{(s^n-1)/(s-1)}
\\
 & > 2^{(s^n-1)/(s-1)} \text{ by inequality (a) above} 
\\
& \ge 2 \quad \text{since $s \ge 2$.}
\end{align*}
Now, by Lemma~\ref{estimate for f} applied to $\bfx(\l)=f_\l^n(\bfc(\l))$, we have
\begin{align*}
  |f_\l^{n+1}(\bfc(\l))|_v & = |f_\l(f_\l^n(\bfc(\l)))|_v
  \\
  & \ge \d_1 |f_\l^n(\bfc(\l))|_v^s \quad \text{by Lemma~\ref{estimate
      for f}}
  \\
  &  \ge \d_1\left( \d_1^{(s^n-1)/(s-1)} |\bfc(\l)|_v^{s^n} \right)^s
  \quad \text{by induction hypothesis,}
  \\
  & = \d_1^{(s^{n+1}-1)/(s-1)} |\bfc(\l)|_v^{s^{n+1}}.
\end{align*}
This concludes the inductive step and the proof of Proposition~\ref{lem:infinite uniform}.
\end{proof}

\begin{prop}
\label{6 uniform large l}
There exist real numbers $L_4\ge 1$, $C_9>0$ and $C_{10}>0$ such that
for all $\l\in V_{L_4}\setminus\{\eta\}$, we have
$$C_9M_n(\l)^{d}\le M_{n+1}(\l)\le C_{10}M_n(\l)^{d},$$
for all $n\in\N$.
\end{prop}

\begin{proof}
We let $L_3$ be defined as in Proposition~\ref{lem:infinite
  uniform} and let $L_5\ge L_3$ satisfy also the inequality
\begin{equation}
\label{(a)}
\d_1^{s+1} L_5^{d_\bfc s} > L_5^{d_\bfc},
\end{equation}
or equivalently $L_5>\d_1^{-(s+1)/(d_\bfc (s-1))}$.

We first claim that   $L_5$ satisfies the following inequality:
  \begin{equation}
    \label{Lv inequality}
    \d_1^{(s^{n+1}-1)/(s-1)} L_5^{d_\bfc s^n} > L_5^{d_\bfc}\quad \text{for all
       $n\in\N$.}
  \end{equation}
  Indeed, for $n=1$, \eqref{Lv inequality} follows from the choice of
  $L_5$ (see inequality \eqref{(a)} above). Now, assume that $n \ge 1$ and \eqref{Lv inequality} holds for
  $n$. Now,
  \begin{align*}
    \d_1^{(s^{n+2}-1)/(s-1)} L_5^{d_\bfc s^{n+1}} & = \d_1^{(s^{n+1}-1)/(s-1)}
    (\d_1^s L_5^{d_\bfc s})^{s^n} \\
    & > \d_1^{(s^{n+1}-1)/(s-1)}  (\d_1^{s+1} L_5^{d_\bfc s})^{s^n} \quad
    \text{since $\d_1 < 1$}
    \\
    & \ge  \d_1^{(s^{n+1}-1)/(s-1)} (L_5^{d_\bfc})^{s^n} \quad \text{by \eqref{(a)}}
    \\
    & > L_5^{d_\bfc} \quad \text{by assumption}.
  \end{align*}
  Hence, by induction we finish the proof of the claim. 

If $\l\in V_{L_5}\setminus\{\eta\}$, then by  
Proposition~\ref{lem:infinite uniform} and inequality \eqref{Lv inequality} we have
\begin{eqnarray*}
|f_{\l}^n(\bfc(\l))|_v\ge & \d_1^{(s^{n+1}-1)/(s-1)} |z(\l)|_v^{d_\bfc s^n}\\
& \ge   \d_1^{(s^{n+1}-1)/(s-1)} L_5^{d_\bfc s^n}\\
& \ge L_5^{d_\bfc} \ge 1,
\end{eqnarray*} 
which means that
$M_n(\l)=|A_{\bfc,n}(\l)|_v$. So, 
\begin{align*}
M_{n+1}(\l) & = |A_{\bfc,n+1}(\l)|_v\\
& = |A_{\bfc,n}(\l)^{d_P}|_v\cdot
\left|P_\l\left(1,\frac{B_{\bfc,n}(\l)}{A_{\bfc,n}(\l)}\right)
\right|_v\\
& =  M_n(\l)^{d}\cdot \left|
  P_\l\left(1,\frac{1}{f_\l^n(\bfc(\l))}\right)\right|_v\quad \text{since $d_P = d$} \\
& = M_n(\l)^{d}\cdot \left| c_{P,0} + \sum_{i=1}^{d_P}
  \frac{\bfc_i(\l)}{f_\l^n(\bfc(\l))^i}\right|_v \\
& =  M_n(\l)^{d}\cdot \left| c_{P,0} + \sum_{i=1}^{d_P}
  \left(\frac{\bfc_i(\l)}{\bfc(\l)^i}\right)\left(\frac{\bfc(\l)}{f_\l^n(\bfc(\l))}\right)^i\right|_v.
\end{align*}
Since we have $|f_\l^n(\bfc(\l))/\bfc(\l)|_v > 2$ whenever $\l\in V_{L_5}\setminus\{\eta\}\subset V_{L_2}\setminus\{\eta\}$ (by \eqref{hypothesis check}), we obtain
\begin{equation*}
 \left|   \sum_{i=1}^{d_P}\left(\frac{\bfc_i(\l)}{\bfc(\l)^i}\right) \left(\frac{\bfc(\l)}{f_\l^n(\bfc(\l))}\right)^i\right|_v
\le  \left(\sum_{i=1}^{d_P} \left|\frac{\bfc_i(\l)}{\bfc(\l)^i}\right|_v\right).
\end{equation*}
Because $d_\bfc >\deg(\bfc_i)/i$ we have
$$\left|\frac{\bfc_i(\l)}{\bfc(\l)^i}\right|_v\to 0\text{ as }\l\to \eta \text{ $v$-adically;}$$
so, there exists $L_4\ge L_5$ such that for $\l\in V_{L_4}\setminus\{\eta\}$  we have 
\begin{equation}
\label{6 22}
\frac{|c_{P,0}|_v}{2} \le
\left|P_\l\left(1,\frac{B_{\bfc,n}(\l)}{A_{\bfc,n}(\l)}\right)\right|_v
\le \frac{3 |c_{P,0}|_v}{2}.
\end{equation}
This concludes the proof of Proposition~\ref{6 uniform large l}.
\end{proof}




\section{Definition of the metrics}
\label{metrics}

We begin with the following lemma that we will use throughout this
section.    
\begin{lemma}
\label{all pullbacks are the same}
Let $\bfw: X \lra
\bP^1$ be a morphism given by $\bfw:=\frac{\bfu}{\bfv}$ where $\bfu,\bfv\in A$
such that $(\bfu,\bfv)=A$. Then the line bundle $\bfw^*\cO_{\bP^1}(1)$ is linearly
equivalent to a multiple of $\eta$. Furthermore, if
$\deg(\bfu)>\deg(\bfv)$, then $\bfw^*\cO_{\bP^1}(1)$ is linearly
equivalent to $\deg (\bfu) \eta$.
\end{lemma}

\begin{proof}
We have that $\bfw^*\cO_{\bP^1}(1)$ equals
$$d_\bfw \eta + \sum_i n_i P_i,$$
where $(P_i,n_i)$ are the zeros $P_i$ with corresponding
multiplicities $n_i$ of $\bfv$. Note that $d_\bfw = (\deg(\bfu) - \deg(\bfv))>0$ if and only if the order of the pole of $\bfu$ at
$\eta$ is larger than the order of the pole of $\bfv$ at $\eta$. On the other hand, $\bfv$ is itself a map from $X$ to $\bP^1$, so
$\sum_i n_i P_i$ is linearly equivalent to $\deg (\bfv) \eta$.  Thus,
$\bfw^*\cO_{\bP^1}(1)$ is linearly equivalent with
$(d_\bfw+ \deg(\bfv )) \eta = \deg (\bfu) \eta$, as desired. 
\end{proof}

Now, let $v\in \Omega_K$ be any
place of $K$.  We put a family of metrics $\| \cdot \|_{v,n}$ on
$\bfc^*\cO_{\bP^1}(1)$ for every positive integer $n$ as 
follows.  Since $\bfc^*\cO_{\bP^1}(1)$ is generated by pull-backs of global
sections of $\cO_{\bP^1}(1)$, it suffices to describe the metric for sections of
the form $z = \bfc^*(u_0t_0 + u_1 t_1)$ where $t_0$ and $t_1$ are the usual coordinate
functions on $\bP^1$ and $u_0,u_1$ are scalars.  For a point $\l \in Y(\bC_v)$, we then define for each $n\in\N$
\begin{equation}
\label{def metric 1}
\|z\|_{v,n}(\l)  := \frac{|u_0 \bfa(\l) + u_1 \bfb(\l)|_v}{\left\{\max \left( |A_{\bfc,n}(\l)|_v, |B_{\bfc,n}(\l)|_v \right)\right\}^{1/d^n}} \text{ if $v$ is nonarchimedean}
\end{equation}
and
\begin{equation}
\label{def metric 2}
\|z\|_{v,n}(\l)  := \frac{|u_0 \bfa(\l) + u_1 \bfb(\l)|_v}{
  \left(|A_{\bfc,n}(\l)|^2_v+ |B_{\bfc,n}(\l)|^2_v \right)^{1/(2
    d^n)} } \text{ if $v$ is archimedean.}
\end{equation}
(Recall that $\deg(A_{\bfc,n})=d_\bfa d^n >\deg(B_{\bfc,n})$.)
Furthermore, we define 
$$\|z\|_{v,n}(\eta)= \vlim_{\l\to\eta} \|z\|_{v,n}(\l) = 
\frac{|u_0|_v}{\left|c_{P,0}^{(d^n-1)/(d^{n+1}-d^n)}\right|_v}.$$
(Note that the leading coefficient of $A_{\bfc,n}$ is $c_{P,0}^{(d^n-1)/(d-1)}c_\bfa^{d^n}$ according to Proposition~\ref{degrees A and B 2}.)

One arrives at \eqref{def metric 1} and \eqref{def metric 2} as
follows.  Let $\Phi_{\bfc,n}: X \lra \bP^1$ be defined by
$\Phi_{\bfc,n }(\lambda) = [A_{\bfc,n}(\lambda): B_{\bfc,n}(\lambda)]$
for $\lambda \not= \eta$ and $\Phi_{\bfc,n }(\eta) = \infty$.  Then, since $A_{\bfc,n}$ has a higher order
pole at $\eta$ than $B_{\bfc,n}$, we see that $\Phi_{\bfc, n}$ sends
$\eta$ to $[1:0]$.  By Lemma \ref{all pullbacks are the same}, we see
then that $(\bfc^* \cO_{\bP^1}(1))^{d^n}$ is isomorphic to
$\Phi_{\bfc,n}^*\cO_{\bP^1}(1)$.  Thus, $\bfc^* t_0^{\otimes d^n}$ and
$\bfc^* t_1^{\otimes d^n}$ are both sections of
$\Phi_{\bfc,n}^*\cO_{\bP^1}(1)$.  Note that $\bfc^* t_0^{\otimes d^n}$
and $\bfc^* t_1^{\otimes d^n}$ have no common zero since $t_0$ and
$t_1$ have no common zero; hence they generate
$\Phi_{\bfc,n}^*\cO_{\bP^1}(1)$ as a line bundle.  Likewise,
$\Phi_{\bfc,n }^*t_0$ and $\Phi_{\bfc,n }^*t_1$ have no common zero
and thus generate $\Phi_{\bfc,n}^*\cO_{\bP^1}(1)$ as a line bundle.
Thus (by \cite[Section II.6]{H}, for example), we have an isomorphism $\tau: (\bfc^* \cO_{\bP^1}(1))^{\otimes
  d^n} {\tilde \lra} \Phi_{\bfc,n}^*\cO_{\bP^1}(1)$, given by $\tau:
\bfc^* t_0^{\otimes d^n} \mapsto \Phi_{\bfc,n }^*t_0$ and $\tau:
\bfc^* t_1^{\otimes d^n} \mapsto \Phi_{\bfc,n }^*t_1$.  Now, for each
place $v$, let $\| \cdot \|'_v$ be the metric on $\cO_{\bP^1}(1)$
given by
\[ \| (u_0 t_0 + u_1 t_1)  \|'_v([a:b]) = \frac{ | u_0 a + u_1 b|_v}
{\max( |a|_v, |b|_v )} \; \text{ if $v$ is nonarchimedean}\] 
and 
\[ \| (u_0 t_0 + u_1 t_1)\|'_v( [a:b]) = \frac{ | u_0 a + u_1 b|_v}
{\sqrt[2]{|a|_v^2 + |b|_v^2}} \text{ if $v$ is archimedean}\] 
(this is the Fubini-Study metric).  Then $\|
\cdot \|_{v,n}$ is simply the $d^n$-th root of $\tau^* \Phi_{\bfc,
  n}^* \| \cdot \|'_v$.  In particular, the adelic metrized line bundle
${\overline L}_n$ given by $\bfc^*\cO_{\bP^1}(1)$ with the metrics $\| \cdot
\|_{v,n}$ is isomorphic to a power of the pullback of a semipositive metrized
line bundle, so it is therefore itself semipositive (see
\cite[Section 2]{zhangadelic}).

\begin{remark}
We also note that for any given model  $\cX$ for $X$ over the ring of
integers $\fo_K$ of $K$, there exists a finite subset $S$ of places of
$K$ depending on $\cX, \bfc$ and $\bff$ such that $\Phi_{\bfc,n}$
extends to a morphism from $\cX$ to $\bP^1$ over the ring of $S$-integers $\fo_S$ of $K$ for all $n.$ From
this we conclude that the family of metrics $\|\cdot\|_{v,n}$ are the same
for all $n$ and all $v\not\in S.$ More precisely, 
we note that $\Phi_{\bfc,n}$ has good reduction for all nonarchimedean primes which do not divide the (constant) resultant of the family $\bff_\l$ and also do not divide the leading coefficients of both $P$ and of $\bfa$. In particular this proves that for each such place $v$ of good reduction, for each $\l\in Y$ and for each integer $n$, we have
$$\max\{|A_{c,n}(\l)|_v, |B_{c,n}(\l)|_v\}=\max\{|\bfa(\l)|_v, |\bfb(\l)|_v\}^{d^n}.$$
Indeed, if $|\bfa(\l)|_v\le |\bfb(\l)|_v$, then dividing  equations \eqref{recursiveA} and \eqref{recursiveB} by $|\bfb(\l)|_v^{d^n}$ and using the fact that $\bfc(\l)$ is integral at $v$ while $\Phi_{\bfc,n}$ has good reduction at $v$ we conclude that
$$\max\left\{\frac{|A_{\bfc,n}(\l)|_v}{|\bfb(\l)|_v^{d^n}}, \frac{|B_{\bfc,n}(\l)|_v}{|\bfb(\l)|_v^{d^n}}\right\}=1.$$
Similarly, if $|\bfa(\l)|_v>|\bfb(\l)|_v$, then $|\bfc(\l)|_v>1$ and since $v$ is a place of good reduction for $\Phi_{\bfc,n}$ we obtain that $|A_{\bfc,n}(\l)|_v=|\bfa(\l)|_v^{d^n}> |B_{\bfc,n}(\l)|_v$. In conclusion, 
$$\max\{|A_{c,n}(\l)|_v, |B_{c,n}(\l)|_v\}=\max\{|\bfa(\l)|_v, |\bfb(\l)|_v\}^{d^n},$$
as claimed.  Now, let $S$ be set of archimedean places along
with the nonarchimedean places $v$ which divide the leading coefficient
of $\bfa$ or $P$ or
divide the constant resultant of the family $\bff_\l$.  Then
\begin{equation*}
\|\cdot \|_{v,n} = \| \cdot \|_{v,0} \text{ for all $v \notin S$ and
  all positive integers $n$} 
\end{equation*}
\end{remark}

\begin{prop}
  \label{finally uniformity}
  For any $v\in \Omega_K$ the sequence of metrics $\| \cdot \|_{v,n}$
  defined above converges uniformly on $X(\C_v)$. 
\end{prop}

\begin{proof}
If $\l=\eta$, the convergence is clear. 
For each $\l\in Y(\C_v)$ we denote by 
$$h_{v,n}(\l):=\max\{|A_{\bfc,n}(\l)|_v,|B_{\bfc,n}|_v\}\text{ if $v$ is nonarchimedean, and}$$
$$h_{v,n}(\l):=\sqrt{|A_{\bfc,n}(\l)|_v^2+|B_{\bfc,n}(\l)|_v^2} \text{ if $v$
  is archimedean.}$$
Then, for a section of the form $z = \bfc^{\ast}(u_0 t_0 + u_1 t_1)$ we
 have 
$$\|z\|_{v,n}(\l) = \frac{|u_0 \bfa(\l) + u_1 \bfb(\l)|_v}{h_{v,n}(\l)^{\frac{1}{d^n}}}.$$ 
To show that  $\log\|\cdot\|_{v,n}$
converge uniformly it suffice to show that  $\frac{\log h_{v,n}(\l)}{d^n}$ converge uniformly for $\l\in Y(\C_v)$. 

Propositions~\ref{lem:finite uniform} and \ref{6 uniform large l} show that there exist positive real numbers $C_{9}$ and $C_{10}$ such that for all $n\ge 0$, we have
\begin{equation}
\label{iterated growth}
C_{9} h_{v,n}(\l)^{d}\le h_{v,n+1}(\l)\le C_{10}h_{v,n}(\l)^{d}.
\end{equation}
In establishing inequality \eqref{iterated growth} at archimedean places $v$, we used the fact that
\begin{equation}
\label{2 proportionality}
\max\{|A_{\bfc,n}(\l)|_v,|B_{\bfc,n}(\l)|_v\}\le h_{v,n}(\l)\le \sqrt{2}\cdot \max\{|A_{\bfc,n}(\l)|_v,|B_{\bfc,n}(\l)|_v\} 
\end{equation}
for all $n\in\N$.   
Taking logarithms in
\eqref{iterated growth} and
dividing by $d^{n+1}$  yields
\[ 
\left|  \frac{1}{d^{n+1}} \log h_{v,n+1}(\l) - \frac{1}{d_1^n} \log
  h_{v,n}(\l) \right| \leq C_{11}/d^{n+1} 
\]
for some positive constant $C_{11}$. 
Thus, by the usual telescoping series argument, we have
\[ \left|  \frac{1}{d^{m}} \log h_{v,m}(\l) - \frac{1}{d^n} \log
h_{v,n}(\l)  \right|  \leq \frac{C_{11}}{d^n} \sum_{i=0}^\infty (1/d^i) =
\frac{C_{11}}{d^n (1-1/d)}  \] 
for any $m > n$.  Since $\frac{C_{11}}{d^n (1-1/d)}$ can be made
arbitrarily small by choosing $n$ large, this gives uniform
convergence for all $\lambda\in Y(C_v)$.  
\end{proof}

For each $v\in \Omega_K$, we let $\|\cdot\|_v$ denote the limit of the
family of metrics $\|\cdot\|_{v,n}$. Proposition~\ref{finally
  uniformity} shows that the adelic metrized 
line bundle ${\overline L} = \left(\bfc^{*}\cO_{\bP^1}(1), \{\|\cdot\|_v\}_{v\in\Omega_K}
\right)$ is semi-positive. Let $\l\in X(\Kbar)$ and choose a
meromorphic section $s$ of $L$ whose support is disjoint from the
Galois conjugates $\l^{[1]}, \ldots, \l^{[\ell]}$ of $\l$ over $K.$
Furthermore, Lemma~\ref{all pullbacks are the same} says that the line
bundle $\bfc^{*}\cO_{\bP^1}(1)$ is isomorphic to $L_\eta^{\otimes d_\bfa}$
where $L_\eta$ is the line bundle determined by the divisor class
containing $\eta$. As in
Section~\ref{subsec:metrized line bundle} we put  
\begin{equation}
  \label{height of bfc}
  h_{\bfc}(\l) :=  \frac{1}{d_\bfa} \sum_{v\in\Omega_K}
  \frac{N_v}{\ell} \sum_{i=1}^\ell   - \log \| s(\l^{[i]}) \|_v . 
\end{equation}
Consequently,  $h_{\bfc} = h_{\overline{L_\eta}}$ is a height function associated to the
metrized line bundle $\overline{L_\eta}$ that corresponds to the
divisor class containing $\eta.$   Now we are ready to give a proof of
Theorem~\ref{thm:uniform convergence}.

\begin{proof}[Proof of Theorem~\ref{thm:uniform convergence}]
Recall that we are given $\bff(x) = P(x)/Q(x)$ of degree $d \ge 2$
over $F$ and a point $\bfc\in F$ such that the sequence $\{\deg
\bff^n(\bfc)\}_{n\ge 0}$ is unbounded. Note that $\bff$ satisfies the
conditions that the resultant $R(\bff)\in K^{*}$ and $d_P \ge d_Q + 2$
(equivalently, the point $x = \infty$ is a superattracting fixed
point for $\bff$).

The first step is to compute the canonical height $\hhat_{\bff}(\bfc)$
of $\bfc \in \bP^1(F)$ associated to the given morphism $\bff$ over $F =
K(X).$ We note that $F$ is a product formula field and moreover the
set of places $\Omega_F$ is in one to one correspondence with the set
of closed points of $X$ over $K$.  Let $e_n = \max\{\deg A_{\bfc,n},
\deg B_{\bfc,n}\}.$ Then,
\begin{align*}
  \hhat_{\bff}(\bfc) & = \lim_{n\to \infty} \frac{1}{d^n} \sum_{P\in
  X} \deg(P) \max \{ - \ord_P(A_{\bfc,n}), - \ord_P(B_{\bfc,n})\} \\
  &  = \lim_{n\to \infty}  \frac{1}{d^n}
  \max\{\deg(A_{\bfc,n}),
  \deg(B_{\bfc, n})\} \quad \text{since $(A_{\bfc,n}, B_{\bfc,n}) = \bfA$}\\
  & = \lim_{n\to \infty} \frac{e_n}{d^n}
\end{align*}
where in the sum $P$ runs over all closed point of $X$ and $\deg(P)$
denote the degree of $P$ over $K.$
Now we put  $g(\l) :=  \hhat_{\bff_\l}(\bfc(\l))/\hhat_{\bff}(\bfc)$
which gives a function on $Y(\Kbar)$. We claim that $g(\l) =
h_\bfc(\l)$ for all $\l.$ Since $h_\bfc$ is a height function
associated to the divisor class containing $\eta$,
Theorem~\ref{thm:uniform convergence} will follow from the claim.

To prove the claim, we first observe that $g(\l)$ is independent of
the choice of the point in the orbit $\cO_\bff(\bfc)=
\{\bff^n(\bfc)\}_{n\ge 0}$. This can be seen as follows: for each
$n \ge 0$ 
$$
\frac{\hhat_{\bff_\l}(\bff_\l^n(\bfc(\l)))}{\hhat_{\bff}(\bff^n(\bfc))}
= \frac{d^n \hhat_{\bff_\l}(\bfc(\l))}{d^n
  \hhat_{\bff}(\bfc)} = g(\l). 
$$ 
In the following, we choose $n$ large enough so that
$\deg(\bff^n(\bfc)) > m$ where $m = m_1 + m_2$ as defined in
\eqref{definition of the m's}. This is possible as
$\deg(\bff^n(\bfc))\to \infty$ when $n\to \infty.$ Replacing
$\bfc$ by $\bff^n(\bfc)$ if necessary, we may assume that $\bfc$
satisfies $d_\bfc = \deg(\bfc) > m.$ Then, according to
Proposition~\ref{degrees A and B   2} we have $e_n = \max\{\deg
A_{\bfc,n}, \deg B_{\bfc,n}\} = d_\bfa d^n.$ Hence,
$\hhat_{\bff}(\bfc) = e_n/d^n = d_\bfa. $ Let $\l^{[1]}, \ldots,
\l^{[\ell]}$ be the Galois conjugates of $\l$ over $K$ and let $s$ be
a section of $\bfc^{\ast} \cO_{\bP^1}(1)$ whose support is disjoint
from  $\l^{[1]}, \ldots, \l^{[\ell]}$. Now we compute
\begin{align*}
d_\bfa h_\bfc(\l) & = \sum_{v\in\Omega_K} \frac{N_v}{\ell} \sum_{i=1}^\ell
- \log \| s(\l^{[i]}) \|_v 
\\
& = \sum_{v\in\Omega_K} \frac{N_v}{\ell} \sum_{i=1}^\ell
\lim_{n\to\infty} \log 
\max\{|A_{\bfc,n}(\l^{[i]})|_v, |B_{\bfc,n}(\l^{[i]})|_v\} -
\log|s(\l^{[i]})|_v \\
& = \sum_{v\in\Omega_K} \frac{N_v}{\ell} \sum_{i=1}^\ell
\lim_{n\to\infty} \log 
\max\{|A_{\bfc,n}(\l^{[i]})|_v, |B_{\bfc,n}(\l^{[i]})|_v\}
\quad\text{by the product formula},
\\
& =
\hhat_{\bff_\l}(\bfc(\l))\quad\text{see~\cite[Theorem~5.59]{Silverman07}}
\\ 
& = \hhat_{\bff}(\bfc) g(\l) \quad \text{by the definition of $g(\l).$}
\end{align*}
As remarked above, we have $\hhat_{\bff}(\bfc) = d_\bfa$. It follows
that $g(\l) = h_\bfc(\l)$ and the proof of Theorem~\ref{thm:uniform
  convergence} is completed. 
\end{proof}



\section{Preperiodic points for families of dynamical systems}
\label{conclusion}

We are ready to prove our main results.
\begin{proof}[Proof of Theorem~\ref{general main}.]
We let   $K$ be a number field such that
$\bff_i$ and also $\bfc_i$ are defined over $K$ (for $i=1,2$). Let
$ h_{\bfc_i}(\l) := \hhat_{\bff_\l}(\bfc_i)/\hhat_{\bff}(\bfc)$ be the
height function defined as in Section~\ref{metrics} for $i=1,2$. As in
the proof of Theorem~\ref{thm:uniform convergence}, $h_{\bfc_i} =
h_{{\overline L}_{\eta,i}}$ is the height function associated to the adelic
metrized line bundle ${\overline L}_{\eta,i} = (L_\eta,
\{\|\cdot\|_{v,i}\}_{v\in \Omega_K})$ where for any $v\in \Omega_K$,
the metric $\{\|\cdot\|_{v,i}\}$ denotes the limits of the metrics
constructed in \eqref{def metric 1} and \eqref{def metric 2}
corresponding to $\bfc_1$ and respectively $\bfc_2$. 

Our hypothesis and Proposition~\ref{finally uniformity} allow us to
use Corollary~\ref{to-use} and therefore conclude the equality of the
two metrics.  That is, 
$$\frac{\hhat_{\bff_{\l,1}}(\bfc_1(\l))}{\hhat_{\bff}(\bfc_1)}=
\frac{\hhat_{\bff_{\l,2}}(\bfc_1(\l))}{\hhat_{\bff}(\bfc_2)}.$$
Therefore,  we have 
$$ \hhat_{\bff_{\l,1}}(\bfc_1(\l)) = 0 \quad \text{if and only
  if}\quad \hhat_{\bff_{\l,2}}(\bfc_1(\l))=0. 
$$ 
This concludes the proof of Theorem~\ref{general main}.
\end{proof}

The following results are easy consequences of Theorem~\ref{general main}.
\begin{cor}
\label{cor 1}
Let $\bfc_i$ and $f_{\l, i}$ be as in Theorem~\ref{general main} for
$i = 1, 2$. Then for each $\l\in\Qbar$,
$\bfc_1(\l)$ is preperiodic for $f_{\l,1}$ if and only if $\bfc_2(\l)$
is preperiodic for $f_{\l,2}$.
\end{cor}

\begin{proof}
Since $\hhat_{f_{\l,i}}(x)=0$  if and only if $x$ is preperiodic for $f_{\l,i}$  (because $f_{\l,i}\in\Qbar(x)$), the conclusion is immediate.
\end{proof}

\begin{proof}[Proof of Theorem~\ref{second main}.]
We let $\bfc_1=\frac{a}{1}$ and $\bfc_2=\frac{b}{1}$; by our
assumption, $\bfc_i$ is a quotient of two functions in $\bfA$ (which generate $\bfA$), and also $\deg(\bff^n(\bfc_i))$ is unbounded as $n\to\infty$. Since
$P_i,Q_i\in\Qbar[x]$ and also $a,b\in\Qbar$, Corollary~\ref{all in
  Kbar} yields that if $a$ (or $b$) is preperiodic under $\bff_{\l}$
(or $\bfg_\l$), then $\l\in\Qbar$. Note that $\bff_{\l}$,
$\bfg_\l$, $\bfc_1(\l)$ and $\bfc_2(\l)$ satisfy the hypothesis of
Theorem~\ref{general main}.  
Using Corollary~\ref{cor 1} we obtain that $a$ is preperiodic under the action of $\bff_{\l}$ if and only if $b$ is preperiodic under the action of $\bfg_{\l}$. 
\end{proof}

The following Corollary generalizes Theorem~\ref{second main} and its proof is identical with the proof of Theorem~\ref{second main}.

\begin{cor}
\label{cor 3}
Let $P_i,Q_i,R_i\in\Qbar[x]$ be nonzero polynomials such that $\deg(P_i)\ge \deg(Q_i)+\deg(R_i)+2$, and let $a,b\in\Qbar$ such that $Q_1(a)$, $R_1(a)$, $Q_2(b)$ and $R_2(b)$ are all nonzero. Let $C$ be a projective nonsingular curve defined over $\Qbar$, let $\eta\in C(\Qbar)$ and let $\bfA$ be the ring of functions on $C$ regular on $C\setminus\{\eta\}$. Let $\Phi,\Psi\in \bfA$ be nonconstant functions. If there exist infinitely many $\l\in C(\Qbar)$ such that both $a$ and $b$ are preperiodic under the action of $\bff_{\l}(x)=P_1(x)/Q_1(x)+\Phi(\l) \cdot R_1(x)$ and respectively of $\bfg_\l(x)=P_2(x)/Q_2(x)+\Psi(\l)\cdot R_2(x)$, then  for all $\l\in C(\C)$, $a$ is preperiodic for $\bff_{\l}$ if and only if $b$ is preperiodic for $\bfg_{\l}$.
\end{cor}

\begin{proof}[Proof of Corollary~\ref{cor 2}.]
By our assumption on the degrees of $P_i$, $Q_i$, $R_i$, $g_i$ for
$i=1,2$ we conclude that the conditions in Theorem~\ref{general main}
is satisfied for $f_{\l,i}(x):=P_i(x)/Q_i(x)+\l R_i(x)$, and $\bfc_i(\l)=g_i(\l)+c_i$ for $i=1,2$. Assume there exist infinitely many $\l\in\Qbar$ such that  $g_i(\l)+c_i$  is preperiodic under the action of $f_{\l,i}(x)$ for $i=1,2$. Using Corollary~\ref{cor 1} we obtain that $g_1(\l)+c_1$ is preperiodic under the action of $f_{\l,1}$ if and only if $g_2(\l)+c_2$ is preperiodic under the action of $f_{\l,2}$. However, $c_1=g_1(0)+c_1$ is preperiodic under $f_{0,1}(x)=P_1(x)/Q_1(x)$, while $c_2=g_2(0)+c_2$ is not preperiodic under $f_{0,2}(x)=P_2(x)/Q_2(x)$. This contradiction proves that indeed there exist at most finitely many $\l\in\Qbar$ such that $g_i(\l)+c_i$ is preperiodic under the action of $f_{\l,i}(x)$ for $i=1,2$.
\end{proof}

Finally, we prove Theorem~\ref{PCF maps}.
\begin{proof}[Proof of Theorem~\ref{PCF maps}.]
We let $\tilde{C}:=C\cup \{\eta\}$ be the projective closure of $C$ in $\bP^2$, where $\eta$ is the point at infinity. We let $A$ be the ring of functions on $\tilde{C}$ which are regular on $C$. Then $\bff_X(z):=f(z)+X\in A[z]$ and also $\bfg_Y(z) :=g(z)+Y\in A[z]$ where $X$ and $Y$ are the corresponding regular functions on $C$, i.e., the functions giving the coordinates of any point. For any critical points $c_1$ and $c_2$ of $f(z)$, respectively of $g(z)$, we let $\bfc_1:=\frac{c_1}{1}\in A$ and $\bfc_2:=\frac{c_2}{1}\in A$. Then all hypotheses of Theorem~\ref{general main} are satisfied.   Therefore, Theorem~\ref{general main} yields that for each $(x,y)\in C(\C)$ we have that $c_1$ is preperiodic for $\bff_x$ if and only if $c_2$ is preperiodic for $\bfg_y$ (note that Corollary~\ref{all in Kbar} yields that each such $(x,y)$ actually lives over $\Qbar$). Repeating the above analysis for each pair of critical points of $f$, respectively of $g$, we conclude that for each point $(x,y)$ on $C$, $f(z)+x$ is PCF if and only if $g(z)+y$ is PCF.  
\end{proof}


\section{Higher dimensional case}
\label{higher}

In this Section we prove Theorem~\ref{first higher}; so we continue with the notation from Theorem~\ref{first higher}. By abuse of notation, we denote by $P(x)$ the polynomial $P\left(\frac{X}{Z},1\right)$ for the variable $x=\frac{X}{Z}$. Similarly we denote by $Q(y)$ the polynomial $Q\left(\frac{Y}{Z},1\right)$ with the variable $y=\frac{Y}{Z}$.

Let $a,b\in \Qbar^*$. For each $n\ge 0$ we let $A_n(\l,\mu),B_n(\l,\mu)\in \Qbar[\l,\mu]$ such that
$$\bff_{\l,\mu}^n([a:b:1])=[A_n(\l,\mu),B_n(\l,\mu):1].$$
More precisely, $A_0=a$ and $B_0=b$, while for each $n\ge 0$, we have
$$A_{n+1}(\l,\mu)=P(A_n(\l,\mu))+\l B_n(\l,\mu)$$ 
and 
$$B_{n+1}(\l,\mu)=Q(B_n(\l,\mu))+\mu A_n(\l,\mu).$$
It is easy to check that $\deg(A_n)=\deg(B_n)=d^{n-1}$ for all $n\ge 1$ (here we use the fact that $d\ge 3$). 
In order to apply our method we consider the following metrics corresponding to any section $u_0t_0+u_1t_1+u_2t_2$ (with scalars $u_i$) of the line bundle $\cO_{\bP^2}(1)$ of $\bP^2$. Using the coordinates $\l=\frac{t_0}{t_2}$ and $\mu=\frac{t_1}{t_2}$ on the affine subset of $\bP^2$ corresponding to $t_2\ne 0$, we get that the metrics $s:=s^{(a,b)}$ are defined as follows: $\|s([t_0:t_1:t_2])\|_{v,n}$ equals
\[\left\{\begin{array}{ccc}
\frac{|u_0t_0+u_1t_1|_v}{\sqrt[2d^{n-1}]{|c_P|_v^{2(d^n-1)/(d-1)}|t_0|_v^{2d^{n-1}}+|c_Q|_v^{2(d^n-1)/(d-1)}|t_1|_v^{2d^{n-1}}}} & \text{if} & t_2=0\\
\frac{|u_0\l+u_1\mu+u_2|_v}{\sqrt[2\cdot d^{n-1}]{\left|A_n\left(\l,\mu\right)\right|_v^2+ \left|B_n\left(\l,\mu\right)\right|_v^2+1}} & \text{if} & [t_0:t_1:t_2]=[\l:\mu:1]
\end{array}\right.,
\]
if $v$ is archimedean, while if $v$ is nonarchimedean, then $
\|s([t_0:t_1:t_2])\|_{v,n}$ equals
\[\left\{\begin{array}{ccc}
\frac{|u_0t_0+u_1t_1|_v}{\sqrt[d^{n-1}]{\max\{|c_P|_v^{(d^n-1)/(d-1)}|t_0|_v^{d^{n-1}},|c_Q|_v^{(d^n-1)/(d-1)}|t_1|_v^{d^{n-1}}\}}} & \text{if} & t_2=0\\
\frac{|u_0\l+u_1\mu+u_2|_v}{\sqrt[d^{n-1}]{\max\{\left|A_n\left(\l,\mu\right)\right|_v, \left|B_n\left(\l,\mu\right)\right|_v,1\}}} & \text{if} & [t_0:t_1:t_2]=[\l:\mu:1]
\end{array}\right.,
\]
where $P(X,0)=c_PX^d$ and $Q(Y,0)=c_QY^d$ (with nonzero constants $c_P$ and $c_Q$ as assumed in Theorem~\ref{first higher}). We note that if we let
$$\tilde{A}_n(t_0,t_1,t_2):=t_2^{d^{n-1}}\cdot A_n\left(\frac{t_0}{t_2},\frac{t_1}{t_2}\right)$$
and
$$\tilde{B}_n(t_0,t_1,t_2)=t_2^{d^{n-1}}\cdot B_n\left(\frac{t_0}{t_2},\frac{t_1}{t_2}\right)$$
then the map
$$\theta_n : [t_0:t_1:t_2]\lra \left[\tilde{A}_n(t_0,t_1,t_2):\tilde{B}_n(t_0,t_1,t_2):t_2^{d^{n-1}}\right]$$
is an endomorphism of $\bP^2$. Indeed, if $t_2=0$, we have
$$\tilde{A}_n(t_0,t_1,0)=c_P^{(d^n-1)/(d-1)}t_0^{d^{n-1}}$$
and
$$\tilde{B}_n(t_0,t_1,0)=c_Q^{(d^n-1)/(d-1)}t_1^{d^{n-1}}$$
which ensures that the above map is well-defined on $\bP^2$.    Thus,
we have an isomorphism  $\tau:  \cO_{\bP^2}(d^{n-1})  {\tilde \lra} \theta^* \cO_{\bP^2}(1)$, given by $\tau:
t_0^{d^{n-1}}  \mapsto \tilde{A}_n(t_0,t_1,t_2)$, $\tau:
t_1^{d^{n-1}}  \mapsto \tilde{B}_n(t_0,t_1,t_2)$ and $\tau:
t_2^{d^{n-1}} \mapsto t_2^{d^{n-1}}$. 

Let $\| \cdot \|'_v$ be the metric on $\cO_{\bP^1}(1)$
given by 
\[ \| (u_0 t_0 + u_1 t_1 + u_2 t_2)  \|'_v([a:b:c]) = \frac{ | u_0 a +
  u_1 b + u_2 c|_v}
{\max( |a|_v, |b|_v, |c|_v )} \; \text{ if $v$ is nonarchimedean}\] 
and 
\[ \| (u_0 t_0 + u_1 t_1 + u_2 t_2)\|'_v( [a:b:c]) = \frac{ | u_0 a +
  u_1 b + u_2 c|_v}
{\sqrt[2]{|a|_v^2 + |b|_v^2 + |c|_v^2}} \text{ if $v$ is archimedean}\] 
(this is the Fubini-Study metric).  We see then that $\|
\cdot \|_{v,n}$ is simply the $d^n$-th root of $\tau^* \theta_n^* \|
\cdot \|'_v$.  Hence, $\| \cdot \|_{v,n}$ is semipositive.  
 
In order to use Corollary~\ref{to-use} we need only prove that the metrics $\|s\|_{v,n}$ (on $\cO_{\bP^2}(1)$) converge uniformly on $\bP^2$ to a metric $\|s\|_v$ on the adelic metrized line bundle ${\overline L}$. Then  we would get that that the height of each point $[\l:\mu:1]$ with respect to ${\overline L}$ is
\begin{equation}
\label{important specialization formula 2}
h_{{\overline L}}([\l:\mu:1])=\frac{\hhat_{\bff_{\l,\mu}}([a:b:1])}{\hhat_{\bff}([a:b:1])}= d\cdot\hhat_{\bff_{\l,\mu}}([a:b:1]),
\end{equation}
since the canonical height of the \emph{constant} point $[a:b:1]\in\bP^2(\Qbar(\l,\mu))$ under the action of $\bff$ is $\frac{1}{d}$. Indeed, for each positive integer $n$, the height of $[A_n(\l,\mu):B_n(\l,\mu):1]\in\bP^2(\Qbar(\l,\mu))$ with respect to the set of places of $\Qbar(\l,\mu)$ (which correspond to the irreducible divisors of $\bP^2_{\Qbar}$ whose function field is identified with $\Qbar(\l,\mu)$) is the same as the total degree of the polynomials $A_n$ and $B_n$, and thus it is $d^{n-1}$. 

Clearly, the metrics $\|s\|_{v,n}$ converge uniformly on the line at infinity from $\bP^2$.  
As before (see Propositions~\ref{lem:finite uniform}, \ref{6 uniform large l}, and \ref{finally uniformity}), we will achieve our goal once we prove that 
$$\frac{M_{n+1}(\l,\mu)}{M_n(\l,\mu)^d}\text{ is uniformly bounded above and below,}$$
where $M_n(\l,\mu):=\max\{|A_n(\l,\mu)|_v,|B_n(\l,\mu)|_v,1\}$.

Let $K$ be a number field containing $a$, $b$ and all coefficients of both $P$ and $Q$. Let $v\in\Omega_K$ be a fixed place. We first observe that there exist real numbers $L_6>1$ and $\d>0$ (depending only on $v$ and on the coefficients of $P$ and $Q$) such that for each $z\in\C_v$ satisfying $|z|_v\ge L_6$, we have
\begin{equation}
\label{above 18}
\min\{|P(z)|_v,|Q(z)|_v\}\ge\d |z|_v^d.
\end{equation}
(Here we use the fact that $\deg(P)=\deg(Q)=d$, which is equivalent with the fact that both $P(X,0)$ and $Q(Y,0)$ are nonzero polynomials.)

Furthermore, there exists a  constant $C_{15}>1$ (depending only on $v$ and on the coefficients of $P$ and $Q$) such that for each $z\in\bC_v$, we have
\begin{equation}
\label{growth of P and Q 2}
\max\{|P(z)|_v,|Q(z)|_v\}\le C_{15}\cdot\max\{1,|z|_v\}^d.
\end{equation}

\begin{lemma}
\label{again in a finite ball}
Let $L$ be any real number larger than $1$. There exist positive real numbers $C_{16}$ and $C_{17}$ depending on $v$, $L$ and on the coefficients of $P$ and $Q$ such that
$$C_{16}\le \frac{ M_{n+1}(\l,\mu)}{M_n(\l,\mu)^d}\le C_{17},$$
for all $n\ge 1$ and for all $\l,\mu\in\C_v$ such that $\max\{|\l|_v,|\mu|_v\}\le L$.
\end{lemma}

\begin{proof}
Clearly, by its construction, $M_n\ge 1$. 
Therefore, using \eqref{growth of P and Q 2} we get
$$|A_{n+1}(\l,\mu)|_v\le |P(A_n(\l,\mu))|_v+|\l|_v\cdot |B_n(\l,\mu)|_v\le C_{15}M_n^d+LM_n\le M_n^d(C_{15}+L)$$
and similarly,
$$|B_{n+1}(\l,\mu)|_v\le |Q(B_n(\l,\mu))|_v+ |\mu|_v\cdot |A_n(\l,\mu)|_v\le M_n^d\cdot (C_{15}+L).$$
This proves the existence of the upper bound $C_{17}$ as in the conclusion of Lemma~\ref{again in a finite ball}.

For the proof of the existence of the lower bound $C_{16}$, we  let $L_7\ge L_6$ be a real number satisfying 
\begin{equation}
\label{above 18 plus}
L_{7}^{d-1}>\frac{2L}{\d}.
\end{equation}

Now we split our analysis into two cases.

{\bf Case 1.} $M_n\le L_7$

In this case, clearly, $\frac{M_{n+1}}{M_n^d}\ge \frac{1}{L_{7}^d}$.

{\bf Case 2.} $M_n>L_{7}$

In this case, without loss of generality, we may assume $|A_n(\l,\mu)|_v=M_n$. Then
\begin{align*}
|A_{n+1}(\l,\mu)|_v\\
& \ge |P(A_n(\l,\mu))|_v - |\l|_v\cdot |B_n(\l,\mu)|_v\\
& \ge \d M_n^d -LM_n\text{ using \eqref{above 18} and  that $|A_n(\l,\mu)|_v>L_7\ge L_6$}\\
& \ge \d M_n^d\cdot \left(1-\frac{L}{\d M_n^{d-1}}\right)\\
& \ge \frac{\d M_n^d}{2}\text{ using \eqref{above 18 plus} and that $M_n>L_{7}$.} 
\end{align*} 
This concludes the proof of Lemma~\ref{again in a finite ball}.
\end{proof}

Let now $L$ be a real number larger than 
$$\max\left\{1,\frac{2|Q(b)|_v}{|a|_v},\frac{2|P(a)|_v}{|b|_v}, \frac{2L_{6}}{\min\{|b|_v,|a|_v\}}, \frac{\d L_{6}^{d-1}}{2}, \frac{2^d}{\d |a|_v^{d-1}},\frac{2^d}{\d |b|_v^{d-1}}\right\}.$$
(Here we use the fact that both $a$ and $b$ are nonzero.) 

\begin{lemma}
\label{outside the finite ball}
If either $|\l|_v>L$, or $|\mu|_v>L$, then 
$$\frac{\d}{2}\le \frac{M_{n+1}}{M_n^d}\le C_{15}+\frac{\d}{2},$$
for each $n\ge 1$.
\end{lemma}

\begin{proof}
Without loss of generality, we may assume $|\l|_v\ge |\mu|_v$; hence $|\l|_v>L$. We note that
$$A_1(\l,\mu)=P(a)+b\l\text{ and }B_1(\l,\mu)=Q(b)+a\mu.$$
Then
\begin{align*}
|A_1(\l,\mu)|_v\\
& \ge |b|_v|\l|_v-|P(a)|_v\\
& \ge \frac{|b|_v|\l|_v}{2}\text{ since $|\l|_v>L>\frac{2|P(a)|_v}{|b|_v}$}\\
& \ge \frac{L\cdot |b|_v}{2}\\
& >L_6\text{ since $L>\frac{2L_6}{|b|_v}$.}
\end{align*}
\begin{claim}
\label{order of growth initially}
$M_n(\l,\mu)^{d-1}\ge \frac{2|\l|_v}{\d}$ for all $n\ge 1$.
\end{claim}

\begin{proof}[Proof of Claim~\ref{order of growth initially}.]
The claim follows by induction on $n$. In the case $n=1$, we have
\begin{align*}
 M_1^{d-1}\\
& \ge |A_1(\l,\mu)|_v^{d-1}\\
& \ge \left(\frac{|b|_v|\l|_v}{2}\right)^{d-1}\\
& \ge \frac{|b|_v|\l|_v}{2}\cdot \left(\frac{|b|_v\cdot L}{2}\right)^{d-2}\text{ since $|\l|_v>L$}\\
& \ge \frac{2|\l|_v}{\d}\cdot \frac{L}{\frac{2^d}{\d |b|_v^{d-1}}}\text{ since $L>1$ and $d-2\ge 1$}\\
& \ge \frac{2|\l|_v}{\d}\text{ since $L>\frac{2^d}{\d |b|_v^{d-1}}$.}
\end{align*}
Now, assume $M_n^{d-1}\ge \frac{2|\l|_v}{\d}$. First we note that since 
$$|\l|_v> L>\frac{\d L_{6}^{d-1}}{2},$$ 
we get that $M_n>L_{6}>1$. So, if $|A_n(\l,\mu)|_v=M_n(\l,\mu)$ then
\begin{align*}
M_{n+1}(\l,\mu)\\
& \ge |P(A_n(\l,\mu))|_v - |\l|_v\cdot |B_n(\l,\mu)|_v\\
& \ge \d M_n(\l,\mu)^d -|\l|_vM_n(\l,\mu)\text{ using \eqref{above 18}}\\
& \ge \d M_n^d\left(1-\frac{|\l|_v}{\d M_n^{d-1}}\right)\\
& \ge \frac{\d}{2}\cdot M_n^d\text{ using the inductive hypothesis.}
\end{align*}
Similarly, if $|B_n(\l,\mu)|_v=M_n(\l,\mu)$ then
\begin{align*}
M_{n+1}(\l,\mu)\\
& \ge |Q(B_n(\l,\mu))|_v-|\mu|_v\cdot |A_n(\l,\mu)|_v\\
&  \ge \d M_n(\l,\mu)^d -|\l|_vM_n(\l,\mu)\text{ using \eqref{above 18} and that $|\l|_v\ge |\mu|_v$}\\
& \ge \d M_n^d\left(1-\frac{|\l|_v}{\d M_n^{d-1}}\right)\\
& \ge \frac{\d}{2}\cdot M_n^d\text{ using the inductive hypothesis.} 
\end{align*}
So, the above inequalities yield that 
$$M_{n+1}\ge M_n\cdot \frac{\d M_n^{d-1}}{2}\ge M_n\cdot |\l|_v\ge M_n\cdot L\ge  M_n,$$
and thus $M_{n+1}^{d-1}\ge \frac{2|\l|_v}{\d}$ as well.
\end{proof}
Furthermore the above proof actually shows the left-hand side of the inequality from the conclusion of our Lemma~\ref{outside the finite ball}, i.e.,
$$M_{n+1}(\l,\mu)\ge \frac{\d}{2}\cdot M_n(\l,\mu)^d.$$ 
We are left to prove the right-hand side of the inequality from our conclusion. For this we use again Claim~\ref{order of growth initially}, and infer that
\begin{align*}
M_{n+1}(\l,\mu)\\
& \le \max\{|P(A_n(\l,\mu))|_v,|Q(B(\l,\mu))|_v\}+\max\{|\l|_v,|\mu|_v\}\cdot M_n(\l,\mu)\\
& \le C_{15}M_n^d+|\l|_v\cdot M_n\text{using \eqref{growth of P and Q 2}}\\
& \le M_n^d\cdot \left(C_{15}+\frac{|\l|_v}{M_n^{d-1}}\right)\\
& \le \left(C_{15}+\frac{\d}{2}\right)\cdot M_n^d,
\end{align*}
as desired.
\end{proof}

Therefore (using Lemmas~\ref{again in a finite ball} and \ref{outside the finite ball}), arguing precisely as in the proof of Proposition~\ref{finally uniformity} we obtain that  $\{\log\|s\|_{v,n}\}$ converges uniformly on $\bP^2$. Indeed, by Lemmas~\ref{again in a finite ball} and \ref{outside the finite ball} we obtain that there exists a positive constant $C_{18}<1$ such that 
$$\frac{1}{C_{18}}\le \frac{M_{n+1}}{M_n^d}\le C_{18}.$$
Hence, there exists a positive constant $C_{19}$  such that for each $m,n\in\N$ with $n>m$
$$\left|\log\|s([\l:\mu:1])\|_{v,n}-\log\|s([\l:\mu:1])\|_{v,m}\right|\le \frac{C_{19}}{d^m}.$$
Hence the sequence of metrics converges uniformly on $\bP^2$.    
This allows us to use Corollary~\ref{to-use} and conclude that for any given $a_i,b_i\in\Qbar^*$ (for $i=1,2$), if there exists a set of points $[\l:\mu:1]$ which is Zariski dense in $\bP^2$ such that for each such pairs $(\l,\mu)$  both $[a_1:b_1:1]$ and $[a_2:b_2:1]$ are preperiodic for $\bff_{\l,\mu}$ then the two sequences of metrics $\left\|s^{(a_i,b_i)}\right\|_{v,n}$ (corresponding to the two starting points $[a_i:b_i:1]$ for $i=1,2$) converge to the \emph{same} metric. Hence, using \eqref{important specialization formula 2}, we obtain the equality of the two canonical heights:
$$\hhat_{\bff_{\l,\mu}}([a_1:b_1:1])=\hhat_{\bff_{\l,\mu}}([a_2:b_2:1]).$$ 
Therefore, for \emph{each} $(\l,\mu)\in\Qbar\times \Qbar$,  $[a_1:b_1:1]$ is preperiodic for $\bff_{\l,\mu}$ if and only if $[a_2:b_2:1]$ is preperiodic for $\bff_{\l,\mu}$. This concludes the proof of Theorem~\ref{first higher}.

\begin{remark}
\label{indeed P2}
If one considers a $2$-parameter family of endomorphisms $\bff_{\l,\mu}$ of $\bP^1$, then for any two starting points $\bfc_1,\bfc_2\in\bP^1$, one expects that there exists a Zariski dense set of parameters $(\l,\mu)$ such that both $\bfc_1$ and $\bfc_2$ are preperiodic for $\bff_{\l,\mu}$. Indeed, for each $i=1,2$ and for each distinct positive integers $m$ and $n$ there exists a curve $C_{i,m,n}$ in the moduli containing all $(\l,\mu)$ such that $\bff_{\l,\mu}^m(\bfc_i)=\bff^n_{\l,\mu}(\bfc_i)$. Thus generically $C_{1,m,n}\cap C_{2,k,\ell}\ne \emptyset$ (for any two pairs of distinct positive integers $(m,n)$ and $(k,\ell)$). Therefore one would expect 
$$\bigcup_{\substack{k,\ell,m,n\in\N\\k\ne \ell\\m\ne n}}C_{1,m,n}\cap C_{2,k,\ell}$$
is Zariski dense in the moduli. Hence the first interesting case when one expects the principle of unlikely intersections in algebraic dynamics holds for a $2$-dimensional moduli is for a family of endomorphisms of $\bP^2$ (as proved in Theorem~\ref{first higher}). 
\end{remark}

\begin{remark}
\label{why Zariski dense}
In the case of a family $\bff_{\l,\mu}$ of endomorphisms of $\bP^2$, the right question is indeed whether there exist a Zariski dense set of points in the moduli for which both starting points $\bfc_1$ and $\bfc_2$ are preperiodic. There are examples when there are infinitely many pairs $(\l,\mu)$ such that both $\bfc_1$ and $\bfc_2$ are preperiodic under $\bff_{\l,\mu}$, but it is not true that $\bfc_1$ is preperiodic under $\bff_{\l,\mu}$ if and only if $\bfc_2$ is preperiodic under $\bff_{\l,\mu}$; this happens when the corresponding points $[\l:\mu:1]$ are not Zariski dense in the moduli $\bP^2$. For example, let
$$\bff_{\l,\mu}\left([X:Y:Z]\right)=[X^3-XZ^2+\l YZ^2:Y^3+\mu XZ^2:Z^3]$$
and $\bfc_1=[0:1:1]$, $\bfc_2=[1:2:1]$. Then $\bfc_1$ is a fixed point for $$\bff_{0,0}\left([X:Y:Z]\right)=[X^3-XZ^2:Y^3:Z^3],$$
while $\bfc_2$ is not preperiodic for the same map. On the other hand, there exist infinitely many $(\l,\mu)\in\Qbar\times \Qbar$ such that both $\bfc_1$ and $\bfc_2$ are preperiodic for $\bff_{\l,\mu}$, but they all lie on a line in the moduli. 

Indeed, let $\mu=\zeta -8$, for any root of unity $\zeta$. Then both $\bfc_1$ and $\bfc_2$ are preperiodic under the action of 
$$\bff_{0,\mu}\left([X:Y:Z]\right)=[X^3-XZ^2:Y^3+(\zeta-8)XZ^2:Z^3].$$
Clearly, $\bfc_1$ is fixed by any map $\bff_{0,\mu}$. On the other hand,  $\bff_{0,\zeta-8}(\bfc_2)=[0:\zeta:1]$, 
which  is preperiodic under any map $\bff_{0,\mu}$ since $\zeta$ is a root of unity and $\bff_{0,\mu}^n(\bfc_2)=\left[0:\zeta^{3^{n-1}}:1\right]$ for any positive integer $n$.
\end{remark}








\begin{thebibliography}{99}
\newcommand{\au}[1]{{#1},}
\newcommand{\ti}[1]{\textit{#1},}
\newcommand{\jo}[1]{{#1}}
\newcommand{\vo}[1]{\textbf{#1}}
\newcommand{\yr}[1]{(#1),}
\newcommand{\pp}[1]{#1.}
\newcommand{\ppx}[1]{#1,}
\newcommand{\pps}[1]{#1;}
\newcommand{\bk}[1]{{#1},}
\newcommand{\inbk}[1]{in: {#1}}
\newcommand{\xxx}[1]{{arXiv:#1.}}



\bibitem{autissier}
P.~Autissier, \emph{Points entiers sur les surfaces arithm\'etiques}, J. Reine.
  Angew. Math. \textbf{531} (2001), 201--235.

\bibitem{Matt-Laura}
\au{M.~Baker and L.~DeMarco}
\ti{Preperiodic points and unlikely intersections}
\jo{Duke Math. J.}
\vo{159} 
\yr{2011} 
\pp{1--29}

\bibitem{BD-preprint}
\au{M.~Baker and L.~DeMarco}
\ti{Post-critically finite polynomials}
preprint.


\bibitem{Baker-Rumely}
\au{M.~Baker and R.~Rumely}
\ti{Potential theory and dynamics on the Berkovich projective line}
AMS Mathematics Surveys and Monographs {\textbf 159} (2010).

%
%
%
%
%
%

\bibitem{bg06}
\au{E.~Bombieri and W.~Gubler}
\ti{Heights in Diophantine Geometry}
\jo{New Mathematical Monograph, Cambridge Univ. Press, Cambridge}
\vo{3}
\yr{2006}

\bibitem{BMZ}
\au{E.~Bombieri, D.~Masser, and U.~Zannier}
\ti{Intersecting a curve with algebraic subgroups of multiplicative groups}
\jo{IMRN}
\vo{20}
\yr{1999}
\pp{1119--1140}




\bibitem{BH88}
\au{B.~Branner and J.~H.~Hubbard}
\ti{The iteration of cubic polynomials. I. The global topology of
  parameter space.}
\jo{Acta Math.}
\vo{160(3-4)}
\yr{1988}
\pp{143-206}


\bibitem{Call-Silverman}
\au{G.~S.~Call and J.~H.~Silverman}
\ti{Canonical heights on varieties with morphisms}
\jo{Compositio Math.}
\vo{89}
\yr{1993}
\pp{163--205}

\bibitem{Carleson-Gamelin}
\au{L.~Carleson and T.~W.~Gamelin}
\ti{Complex dynamics}
Springer-Verlag, New York, 1993.

\bibitem{CL}
A.~Chambert-Loir, \emph{Mesures et \'equidistribution sur les espaces de
  {B}erkovich}, J. Reine Angew. Math. \textbf{595} (2006), 215--235.

\bibitem{DF}
\au{R.~Dujardin and C.~Favre}
\ti{Distribution of rational maps with a preperiodic critical point}
\jo{Amer. J. Math}
\vo{130}
\yr{2008}
\pp{979--1032}

\bibitem{favre-rivera}
  \au{C.~Favre and J.~Rivera-Letelier}
  \ti{Th\'eor\`eme d'\'equidistribution de Brolin en dynamique
    $p$-addique}
  \jo{C. R. Math. Acad. Sci. Paris}
  \vo{339(4)}
  \yr{2004}
  \pp{271-276}

\bibitem{FRL}
\au{C.~Favre and J. Rivera-Letelier}
\ti{\'Equidistribution quantitative des points des petite hauteur sur la droite projective}
\jo{Math. Ann.}
\vo{355}
\yr{2006}
\pp{311--361}




\bibitem{GT-IMRN}
\au{D.~Ghioca, T.~J.~Tucker and S.~Zhang}
\ti{Towards a dynamical Manin-Mumford Conjecture}
\jo{IMRN}
\vo{22}
\yr{2011}
\pp{5109--5122} 

%
%

\bibitem{prep}
\au{D.~Ghioca, L.-C.~Hsia and T.~J.~Tucker}
\ti{Preperiodic points for families of polynomials}
\jo{Algebra $\&$ Number Theory}
to appear (2012), 26 pages.


\bibitem{Habegger}
\au{P.~Habegger}
\ti{Intersecting subvarieties of abelian varieties with algebraic
  subgroups of complementary dimension} 
\jo{Invent. Math.}
\vo{176}
\yr{2009}
\pp{405--447}

\bibitem{H}
R.~Hartshorne, \emph{Algebraic geometry}, Springer-Verlag, New York, 1977.



\bibitem{ingram10}
\au{P.~Ingram}
\ti{Variation of the canonical height for a family of polynomials} 
\jo{J. Reine. Angew. Math}
to appear (2012), 23 pages.



\bibitem{lang}
\au{S.~Lang}
\ti{Fundamental of Diophantine Geometry}
Springer-Verlag, New York, 1983.


\bibitem{M-Z-1}
\au{D.~Masser and U.~Zannier}
\ti{Torsion anomalous points and families of elliptic curves}
\jo{C. R. Math. Acad. Sci. Paris}
\vo{346}
\yr{2008}
\pp{491--494}

\bibitem{M-Z-2}
\au{D.~Masser and U.~Zannier}
\ti{Torsion anomalous points and families of elliptic curves}
\jo{Amer. J. Math.}
\vo{132}
\yr{2010}
\pp{1677--1691}

\bibitem{M-Z-3}
\au{D.~Masser and U.~Zannier}
\ti{Torsion points on  families of squares of elliptic curves}
\jo{Math. Ann.} 
\yr{2012}
\vo{352}
\pp{453--484}

\bibitem{STP}
J.~Pi{\~n}eiro, L.~Szpiro, and T.~Tucker, \emph{Mahler measure for dynamical
  systems on $\mathbb{P}^{1}$ and intersection theory on a singular arithmetic
  surface}, Geometric methods in algebra and number theory (F.~Bogomolov and
  Y.~Tschinkel, eds.), Progress in Mathematics 235, Birkh{\"a}user, 2004,
  pp.~219--250.

\bibitem{Pink}
\au{R.~Pink}
\ti{A common generalization of the conjectures of Andr\'{e}-Oort,
  Manin-Mumford, and Mordell-Lang} 
preprint, 2005.

%
%

\bibitem{Silverman83}
  J.~Silverman, \emph{Heights and the specialization map for families
    of abelian varieties.}, J. Reine Angew. Math. \textbf{342} (1983),
  197--211. 

\bibitem{Silverman07}
  J.~Silverman, The Arithmetic of Dynamical Systems, \emph{GTM 241},
  Springer-Verlag, 2007.


\bibitem{tate83}
  J.~Tate, \emph{Variation of the canonical height of a point
    depending on a parameter}, Amer.~J.~Math. \textbf{105} (1983), no.~1, 287-294.
  
\bibitem{Yuan}
X.~Yuan, \emph{Big line bundles over arithmetic varieties}, Invent. Math.
  \textbf{173} (2008), no.~3, 603--649.

\bibitem{YZ}
\au{X.~Yuan and S.~Zhang}
\ti{Calabi Theorem and algebraic dynamics}
preprint (2010), 24 pages.

\bibitem{Zannier}
U.~Zannier, \emph{Some problems of unlikely intersections in arithmetic and
  geometry}, Annals of Mathematics Studies, vol. 181, Princeton University
  Press, Princeton, NJ, 2012, With appendixes by David Masser.


\bibitem{zhangadelic}
S.~Zhang, \emph{Small points and adelic metrics}, J.~Algebraic Geometry
  \textbf{4} (1995), 281--300.

\end{thebibliography}
\end{document}